\documentclass[11pt,reqno]{amsart} 
\usepackage[a4paper,left=24mm,right=24mm,top=32mm,bottom=30mm]{geometry} 
\usepackage{epic,eepic,verbatim,amsmath,amssymb,enumitem} 
\usepackage{graphicx} 
\usepackage{color} 
\usepackage[pdftex]{hyperref} 
\usepackage{subfigure}

\newcommand{\eps}{\varepsilon}             
\newcommand{\pd}{\partial}                 
\renewcommand{\d}{{\rm{d}}}                
\newcommand{\dg}{\, \d \sigma}    
\newcommand{\dgx}{\, \d \sigma(x)}    
\newcommand{\dx}{\, {\rm{d}} x}            
\newcommand{\D}{\nabla}                    
\newcommand{\R}{\mathbb{R}}                
\newcommand{\N}{\mathbb{N}}                
\newcommand{\Z}{\mathbb{Z}}                
\newcommand{\co}{\rightarrow}              
\newcommand{\SL}{\Delta_{\Gamma}}          
 
\newcommand{\fu}{f_u}   
\newcommand{\fv}{f_v}   
\newcommand{\qu}{q_u}   
\newcommand{\qv}{q_v}   
\newcommand{\qV}{q_V}   
\newcommand{\overbar}[1]{\mkern
  1.5mu\overline{\mkern-1.5mu#1\mkern-1.5mu}\mkern 1.5mu}

\DeclareMathOperator{\const}{const\!}

\renewenvironment{proof}[0] {\noindent{\em Proof.}}{\hfill \qed\\[1ex] }

\theoremstyle{plain}
\numberwithin{equation}{section}
\newtheorem{lemma}{Lemma}[section]
\newtheorem{theorem}[lemma]{Theorem}
\newtheorem{proposition}[lemma]{Proposition}

\newtheorem{corollary}[lemma]{Corollary}
\theoremstyle{definition}
\newtheorem{remark}[lemma]{Remark}

\begin{document} 

\title[Symmetry breaking in a reaction--diffusion model]
{Symmetry breaking in a bulk--surface reaction--diffusion model for
  signaling networks}

\author[A. R\"atz]{Andreas R\"atz}
\address{Andreas R\"atz, Technische Universit\"at Dortmund,
Fakult{\"a}t f{\"u}r Mathematik,
Vogelpothsweg 87,
D-44227 Dortmund}
\email{andreas.raetz@tu-dortmund.de}

\author[M. R\"oger]{Matthias R\"oger}
\address{Matthias R{\"o}ger, Technische Universit\"at Dortmund,
Fakult{\"a}t f{\"u}r Mathematik,
Vogelpothsweg 87,
D-44227 Dortmund}
\email{matthias.roeger@tu-dortmund.de}



\begin{abstract}
Signaling molecules play an important role for many cellular functions. We investigate here a general system of two membrane reaction--diffusion equations coupled to a diffusion equation inside the cell by a Robin-type boundary condition and a flux term in the membrane equations. A specific model of this form was recently proposed by the authors for the GTPase cycle in cells. We investigate here a putative role of diffusive instabilities in cell polarization. By a linearized stability analysis we identify two different mechanisms. The first resembles a classical Turing instability for the membrane subsystem and requires (unrealistically) large differences in the lateral diffusion of activator and substrate. The second possibility on the other hand is induced by the difference in cytosolic and lateral diffusion and appears much more realistic. We complement our theoretical analysis by numerical simulations that confirm the new stability mechanism and allow to investigate the evolution beyond the regime where the linearization applies.
\end{abstract}

\date{\today}
\maketitle

\section{Introduction}
In numerous biological processes the emergence and maintenance of polarized states in the form of heterogeneous distributions of chemical substances (proteins, lipids) is essential. Such symmetry breaking for example precedes the formation of buds in yeast cells, determines directions of movement, or mediates differentiation and development of cells. Polarized states in biological cells often arise in response to external signals, typically from the outer cell membrane. Transport processes and interacting networks of diffusing and reacting substances both within the cell and on the cell membrane then amplify and process such signals. The distribution of small GTPase molecules in eukaryotic cells presents one example of a complex system with polarization and motivates the present paper. Such molecules can be in an active and in an inactive state. Activation and deactivation typically occurs at the cell membrane and is catalyzed by specific enzymes. In addition to the activation-deactivation cycle GTPase molecules (in its inactive form) shuttle between the membrane and the cytosol, i.e. the inner volume of the cell, by attachment to and detachment from the membrane. These properties induce the specific form of a coupled volume (bulk) and surface reaction--diffusion system. The goal here is to investigate a possible contribution of diffusive instabilities to cell polarization in such coupled systems. 
 
Different deterministic continuous models have been used to investigate polarization of cells (see for example the review \cite{JiEd11}). We consider here a particular class of models that takes the form of a reaction-diffusion system on the membrane coupled to a diffusion process in the interior of the cell. This model has been introduced in \cite{RaRo12} where a reduction to a non-local reaction-diffusion system, involving only membrane variables, has been investigated. A key question in such systems is whether a Turing-type mechanism may contribute to the polarization of cells, with activated GTPase and inactive GTPase representing self-activator and substrate, respectively. Turing-type instabilities however require large differences in the diffusion constants for activator and substrate (or inhibitor). The lateral diffusion on the membrane for active and inactive GTPase on the other hand is in general of comparable size and therefore a Turing instability appears at first glance unrealistic. On the other hand, cytosolic diffusion in cells is typically much faster than lateral diffusion and might induce the necessary difference in diffusion. It is therefore tempting to hypothesize that in such coupled 2D and 3D reaction--diffusion systems diffusive instabilities can in fact contribute to cell polarization.

We will in the following represent the cytosolic volume and the membrane of a cell by a bounded, connected, open domain $B \subset
\R^3$ in space and its two-dimensional boundary $\Gamma := \pd B$
respectively. We assume that $\Gamma$ is given by a smooth, closed
surface and denote by $\nu$ the outer unit normal of $B$ on $\Gamma$. Further we fix a time
interval of observation  $I:=[0,T] \subset \R$ and consider smooth
functions $V:\overbar B \times I \to\R$, $u,v: \Gamma \times I \to
\R$ (representing the cytosolic inactive, membrane-bound active, and
membrane-bound inactive GTPase, respectively) that satisfy the
coupled reaction--diffusion system (stated in a non-dimensional form)
\begin{align}
  \label{eq:V}
  \pd_t  V &=  D \Delta V \quad &&\text{in} \quad
  B  \times  I,\\
  \label{eq:u}
  \pd_t  u &= \SL  u +  \gamma f(u,v)\quad &&\text{on} \quad
  \Gamma \times  I,\\
  \label{eq:v}
  \pd_t  v &= d \SL  v +  \gamma (- f(u,v) + q(u,v,V)) \quad &&\text{on}
  \quad  \Gamma \times  I,\\  
  -D \nabla V \cdot \nu &= \gamma q(u,v,V)  \quad &&\text{on} \quad \Gamma\times I.\label{eq:flux1}
\end{align}
Here $f$ and $g$ represent the activation/inactivation processes and
$q$ describes attachment/detachment at the membrane. In the appendix
we present as specific example the mathematical model for the GTPase cycle from \cite{RaRo12} with explicit choices for $f$ and
$q$ that we also use in our numerical simulations below. The parameter $\gamma > 0$ is a non-dimensional parameter that is related to the spatial scale of the
cell. The coupling of bulk and surface processes in
\eqref{eq:V}--\eqref{eq:flux1} is given in form of a Robin-type
boundary condition. The system is complemented by initial conditions at time $t=0$,
\begin{align*}
	&V(\cdot, 0) \,=\, V_0,\quad v(\cdot, 0) \,=\, v_0,\quad
	u(\cdot, 0) \,=\, u_0,\\
	&V_0\,:\, B\,\to\, \R,\quad v_0,u_0\,:\,\Gamma\,\to\,\R.
\end{align*}
We remark that the system \eqref{eq:V}--\eqref{eq:flux1} automatically satisfies conservation of total mass, i.e.
\begin{gather*}
	M(t)  \,:=\, \int_B V(x,t)\dx + \int_\Gamma (u+v)(x,t) \dgx \,=\, \const,
\end{gather*}
where $\dg$ denotes integration with respect to the surface area measure.

In this contribution we investigate the possibility of diffusive
instabilities for systems of the form
\eqref{eq:V}--\eqref{eq:flux1}. We present a linear stability analysis
and numerical simulations. We find two possible scenarios for a
diffusive instability of spatially homogeneous stationary
states. The first needs large differences in lateral diffusion
for $u$ and $v$ (i.e. a coefficient $d\gg 1$) and resembles a classical Turing instability in the $u,v$ variables. The second mechanism on the other hand does also occur for equal lateral diffusion constants $d=1$ and is rather  based on the different diffusion constants for $u$ and $V$ and therefore on the coupling of bulk and surface equations. As cytosolic diffusion is typically by a factor hundred faster than lateral diffusion, this scenario is much more realistic in the application to signaling networks. In Section \ref{sec:turingNonLocalA} we compare the stability of the full system to its reduction in the formal limit $D\to\infty$. The latter leads to a non-local two-variable system on the membrane that has been analyzed in \cite{RaRo12}
(see also \cite{RaRo13}). There we have only covered the first more
`classical' instability mechanism and have not included a complete
characterization of diffusive instabilities. Here we show that -- in coincidence with the case $D<\infty$ -- we again have the same alternative scenarios for diffusive instabilities. Some specific properties of the second instability mechanism are easier to characterize in the reduction. In particular we find that the second scenario
is different from a standard Turing-type instability as in this case the concentration of activated GTPase in a single spot (most typical in most examples of cell polarization) is always preferred
independent of variations in the parameter values. This robustness makes the second instability mechanism an even more attractive explanation for polarization.  In
Section \ref{sec:numericsFull} we present numerical simulations for 
specific versions both of the full system \eqref{eq:V}--\eqref{eq:flux1} and of the reduced system. The simulations
confirm the instability criteria derived for the linearised system and
allow to investigate the time-evolution after the onset of
heterogeneities and beyond the regime governed by the
linearization. It turns out that even for simple choices of the
constitutive relations $f$ and $q$ the system exhibits a rich 
behavior. In the final Section \ref{sec:discuss} we discuss the results of the paper and in particular comment on the term `diffuse instability' in the present context and with respect to the second instability mechanism.

The most specific property of the model considered here is the coupling of bulk and surface reaction--diffusion systems. Such coupled systems often arise in cell biology where enzymatic processes on intracellular membranes play a role, see \cite{NoGaChReScSl07} and the references therein. Surface--bulk reaction--diffusion or convection--diffusion systems also arise in the modeling of surfactants on two-fluid interfaces \cite{TeLiLoWaVo09}. Coupled surface--bulk systems have been studied intensively over the last decades by numerical simulations, see for example \cite{LeRa05,TeLiLoWaVo09, HeHL12, NoGaChReScSl07, ElRa13} and the references therein. \\
In \cite{LeRa05} a model that is similar to ours and that describes a two-variable diffusion system in a volume coupled to a reaction system on the boundary has been studied. Here the authors provide numerical simulations and a linear stability analysis. They show the existence of Turing instabilities in their model, even for equal bulk diffusion constants. The main difference to our model is that in \cite{LeRa05} both activator and substrate diffuse in the bulk and that no diffusion on the membrane surface is considered.
\section{Stability Analysis For the Full System}\label{sec:turingFullA}

We consider in the following the spatially coupled reaction--diffusion system \eqref{eq:V}--\eqref{eq:flux1} for spherical cell shapes and investigate the possibility of diffusive instabilities of a homogeneous stationary state. Because of the different domains of definition of $V$ and $u,v$ we cannot apply standard conditions for (in)stability and therefore will derive appropriate criteria in this section. Similar to a classical Turing instability we consider a spatially homogeneous stationary state and require that this state is (a) stable against perturbations of the membrane quantities $u,v$ that are spatially homogeneous \emph{on the membrane} and perturbations of the bulk variable $V$ that are radially symmetric (this additional restriction follows already by the first) and (b) unstable with respect to general perturbations. Such property represents a diffusive instability and a symmetry breaking in the sense that the radially symmetric evolution loses its stability as it approaches such a stationary point.

For the constitutive relations $f,q$ we assume that
\begin{align}
	\partial_v f \,&\geq\, 0, \quad
	\partial_v q\,\leq \, 0, \quad \partial_v q\,\leq\, \partial_u q,\quad  \partial_V q \,\geq\, 0,\label{cdt:q}
\end{align}
which are for the application to the GTPase cycle natural conditions
with respect to the interpretation of $f$ as activation rate and $q$ as the flux induced by ad-
and desorption of GTPase at the membrane. As we are interested in symmetry breaking we consider in the following a spatially homogeneous stationary
state $(u_*,v_*,V_*)\in \R_+^3 := \{(x_1,x_2,x_3) \in \R^3 \,:\, x_i >
0,\, i=1,2,3\}$ of \eqref{eq:V}--\eqref{eq:flux1}, which is equivalent to the conditions
\begin{align}
	f(u_*,v_*)\,&=\,0, \label{eq:equi-f}\\
	q(u_*,v_*,V_*)\,&=\,0. \label{eq:equi-q}
\end{align}
For convenience we introduce the following notation,
\begin{align}
	\fu  \,&:=\, \partial_u f(u_*,v_*),\quad \fv  \,:=\, \partial_v f(u_*,v_*),\notag\\
	\quad \qu  \,&:=\, \partial_u q(u_*,v_*,V_*),\quad  \qv  \,:=\, \partial_v q(u_*,v_*,V_*),\quad \qV \,:=\, \partial_V q(u_*,v_*,V_*). \label{eq:coeff}
\end{align}
We assume that in $(V_*,u_*,v_*)$ we have strict inequalities 
\begin{align}
	f_v \,&>\, 0, \quad
	q_v\,< \, 0, \quad q_V\,>\, 0,\label{cdt:q-star}
\end{align}

The linearization of \eqref{eq:V}--\eqref{eq:flux1} in $(V_*,u_*,v_*)$ is given by the system
\begin{align}
  \label{eq:V-lin}
  \pd_t  V &=  D \Delta V \quad &&\text{in} \quad
  B  \times  I,\\
  \label{eq:u-lin}
  \pd_t  u &= \SL  u + \gamma \big(\fu u +  \fv v\Big) \quad &&\text{on} \quad
  \Gamma \times  I,\\
  \label{eq:v-lin}
  \pd_t  v &= d \SL  v + \gamma \big( (-\fu +\qu )u +(-\fv +\qv )v +\qV V\big) \quad &&\text{on}
  \quad  \Gamma \times  I,\\  
  -D \nabla V \cdot \nu &= \gamma \big( \qu u +\qv v +\qV V\big)  \quad &&\text{on} \quad \Gamma\times I\label{eq:flux1-lin}
\end{align}
for unknowns $V:B\times\R\to \R$, $u,v:\Gamma\times\R\to\R$, together with a constraint on the initial conditions, due to the mass conservation property,
\begin{align}
	\int_B V(x,0)\dx + \int_{\Gamma} \big(u(x,0) + v(x,0)\big)\dgx \,&=\, 0. \label{eq:mass-lin}
\end{align}
In the following we assume that no inner membranes are present and assume a spherical shape of the cell, i.e. 
we choose $B = B_1(0)$ and $\Gamma=\partial B =S^2$. This allows in the subsequent stability analysis to use separated variables: we introduce polar coordinates and represent $x\in B$ as $x = r y$ with $y \in
S^2, r \in [0,1]$. We further fix an orthonormal basis  $\{\varphi_{lm}\}_{l\in\N_0,m\in\Z, |m|\leq l}$ of $L^2(\Gamma)$ given by spherical harmonics with
\begin{align}
	-\SL \varphi_{lm} \,&=\, l(l+1)\varphi_{lm} \quad\text{ on }\Gamma \label{eq:sh-1}
\end{align}
and remark that $\varphi_{00}$ is constant on $\Gamma$.
We then consider the following ansatz for solution of the linearized system \eqref{eq:V-lin}--\eqref{eq:flux1-lin}, 
\begin{align}
  u(y,t) &= \sum_{l\in\N_0,m\in\Z, |m|\leq l} u_{lm}(t)\varphi_{lm}(y),
   \\
  v(y,t) &= \sum_{l\in\N_0,m\in\Z, |m|\leq l} v_{lm}(t)\varphi_{lm}(y),
   \\
  V(ry,t) &= \sum_{l\in\N_0,m\in\Z, |m|\leq l} V_{lm}(t)\psi_{lm}(r)\varphi_{lm}(y),
\end{align}
with $u_{lm},v_{lm},V_{lm}:\R\to \R$, $\psi_{lm}:[0,1]\to \R$, $y\in
\Gamma, 0\leq r\leq 1$ (for a similar approach see \cite{LeRa05}). We deduce from
\eqref{eq:V-lin}--\eqref{eq:flux1-lin}, by taking the $L^2(\Gamma)$
scalar product with $\varphi_{lm}$,
\begin{align}
  u_{lm}' &= -  l(l + 1) u_{lm} +  \gamma (\fu u_{lm} + \fv v_{lm}), \label{eq:stab_u}
  \\
  v_{lm}' &= -  d l(l + 1) v_{lm} + \gamma \Big( (-\fu +\qu )u_{lm} + (-\fv +\qv )v_{lm} + \qV \psi_{lm}(1)V_{lm}\Big), \label{eq:stab_v}
  \\
   V_{lm}'(t)\psi_{lm}(r) &= D V_{lm}(t)\left(\psi_{lm}''(r) + \frac2{r}\psi_{lm}'(r) - \frac1{r^2}l(l + 1)\psi_{lm}(r)\right), \label{eq:stab_V}
  \\
  -D V_{lm} \psi_{lm}'(1) &= \gamma (\qu u_{lm} + \qv v_{lm} + \qV \psi_{lm}(1)V_{lm}). \label{eq:stab_q}
\end{align}
From \eqref{eq:stab_V} we obtain that
\begin{align}
	V_{lm}(t)\,=\, \bar{B}_{lm}e^{\omega_{lm}t}, \quad
        \bar{B}_{lm} \in \R, \quad
        \omega_{lm} \in \R, \label{eq:stab_V1}
\end{align}
and $V_{lm}$ is either identically zero or does nowhere vanish.

In the following we first restrict ourselves to the case $V_{lm}\neq 0$. 
We deduce that
\begin{align}
	0\,&=\, r^2 \psi_{lm}''(r) + 2r\psi_{lm}'(r) - \left(l(l+1) + \frac{\omega_{lm}}{D}r^2\right)\psi_{lm}(r).
	\label{eq:psi-lm}
\end{align}
If in addition $\omega_{lm}> 0$, the latter equation implies that 
\begin{align}
  \psi_{lm}(r) \,&=\, \alpha_{lm} i_l\left(\sqrt{\frac{\omega_{lm}}{D}} r \right), \alpha_{lm}\in\R, \label{eq:stab_V12}\\
  i_l(r) \,& =\, \sqrt{\frac{\pi}{2r}} I_{l + \frac1{2}}(r), \notag 
\end{align}
where $I_{l+\frac{1}{2}}$ denotes the respective modified Bessel functions of first kind. 
\\
In the case $\omega_{lm}=0$ we obtain instead
\begin{align}
	\psi_{lm}(r) \,&=\, \alpha_{lm} r^l, \alpha_{lm}\in\R. \label{eq:psi-lm-0}
\end{align}

We derive from \eqref{eq:stab_u}, \eqref{eq:stab_v}, \eqref{eq:stab_V1} and \eqref{eq:stab_q}  the linear ODE system
\begin{align}
  \label{eq:stab_alm}
  u_{lm}' &= \big(-  l(l + 1) + \gamma \fu \big)u_{lm} +  \gamma \fv v_{lm},\\
  \label{eq:stab_blm}
  v_{lm}' &=  -\gamma \fu  u_{lm}  - \big(  d l(l + 1)+\gamma \fv  \big) v_{lm} -D \psi_{lm}'(1)V_{lm},\\
  V_{lm}' &= \omega_{lm} V_{lm}, \label{eq:stab_Vlm}
\end{align}
coupled to an algebraic equation
\begin{align}
  0 &= \gamma (\qu u_{lm} + \qv v_{lm}) + \big(\gamma \qV \psi_{lm}(1)  +D \psi_{lm}'(1)\big)V_{lm}, \label{eq:stab_Blm}
\end{align}
that determines in the case $V_{lm}\neq 0$ together with \eqref{eq:stab_V12} the value of $\omega_{lm}$.
The linear stability analysis then reduces to an analysis of the eigenvalue equation coupled to an algebraic condition. We obtain that an eigenvalue $\omega$ with nonnegative real part exists if and only if first $\omega=\omega_{lm}\in \R^+_0$ and second $\omega_{lm}$ satisfies
\begin{align}
	0 \,&\overset{!}{=} G_l(\omega_{lm}) \notag\\
	&:=\, \gamma \qV \Big(\omega_{lm}^2 + \big((d+1)l(l+1) +(-\fu +\fv )\gamma\big)\omega_{lm}+dl^2(l+1)^2 +\gamma l(l+1)(-d\fu +\fv )\Big) \notag\\
	&\qquad +\kappa_{D,l}(\omega_{lm}) \Big(\omega_{lm}^2 +\big((d+1)l(l+1) +(-\fu +\fv )\gamma\big)\omega_{lm}+			\notag \\
	&\qquad\qquad\qquad\qquad\qquad\qquad  +dl^2(l+1)^2 +\gamma l(l+1)(-d\fu +\fv )\Big)\notag\\
	&\qquad +\kappa_{D,l}(\omega_{lm}) \Big(-\gamma \qv \big(l(l+1)+\omega_{lm}\big)+\gamma^2\big(\fu \qv -\fv \qu \big)\Big)\label{eq:eigenvalue-eq}
\end{align}
with
\begin{align}
  \kappa_{D,l}(\omega)\,&:=\, \frac{D \psi_{lm}'(1)}{\psi_{lm}(1)}\,=\,
	  D\left(\frac{ri_l'(r)}{i_l(r)}\right)\Big|_{r=\sqrt{\frac{\omega}{D}}}, 
  \label{eq:2.24}
\end{align}
where the last equality follows for $\omega>0$ from  \eqref{eq:stab_V12} and for $\omega_{lm}=0$ from \eqref{eq:psi-lm-0} and \eqref{eq:kappa0} below.

In the case $V_{lm}=0$, which is equivalent to $\bar{B}_{lm}=0$, the system \eqref{eq:stab_u}--\eqref{eq:stab_q} is overdetermined. We obtain that the parameters have to satisfy the equation
\begin{align}
	(d-1)l(l+1) q_u \,&=\, \gamma(-f_uq_v + f_v q_u)\frac{1}{q_v}(q_u-q_v) \label{eq:cdt-Vlm0}
\end{align}
and that under this condition any eigenvalue $\omega$ corresponding to the linearized system \eqref{eq:stab_u}--\eqref{eq:stab_q} is given by
\begin{align}
	-q_u\omega\,&=\, \gamma (-f_uq_v + f_vq_u ) + dl(l+1)q_u. \label{eq:om-Vlm0}
\end{align}
Due to the condition \eqref{eq:cdt-Vlm0} the case $V_{lm}=0$ is only relevant for a small -- nowhere open -- subset of the parameter space. Therefore this case cannot contribute to a robust mechanism. 
%
%
%
\subsection{Asymptotic stability with respect to spatially homogeneous perturbations}
Here we would like to consider perturbations $u,v,V$ of a stationary
state $(u_*,v_*,V_*)\in \R_+^3$ such that the membrane quantities
$u,v$ are spatially homogeneous but $V$ is allowed to be heterogeneous. We would like to characterize the stability of our system under such perturbations. In the ansatz
above the restriction to spatially homogeneous $u,v$ means that
$u_{lm}=v_{lm}=0$ for all $l\geq 1$. By \eqref{eq:stab_blm} and
$\psi_{lm}'(1)>0$ (see for example \cite[(10.51.5)]{DLMF}, \cite{OLBC10}) we deduce that also
$V_{lm}=0$ for all $l\geq 1$, and in particular that $V$ is radially symmetric. It
remains to study the condition \eqref{eq:eigenvalue-eq} for $l=0$.
\begin{proposition}\label{prop:stab}
A necessary and sufficient condition for the stability of \eqref{eq:V}--\eqref{eq:flux1} in $(u_*,v_*,V_*)$ under
perturbations that are spatially homogeneous in the $u,v$ variables is that 
\begin{align}
	0\,&<\, \frac{1}{3}(\fu \qv -\fv \qu ) + \qV (\fv -\fu ). \label{eq:case1-stab2}
\end{align}
In this case
\begin{align}
	f_v \,&>\, f_u \label{eq:fu-fv-full}
\end{align}
holds.
\end{proposition}
\begin{proof}
As remarked above we have to consider $l=0$. We first restrict ourselves to the case $V_{00}\neq 0$. Then the system \eqref{eq:V}--\eqref{eq:flux1} is linearly asymptotically stable in $(u_*,v_*,V_*)$ if and only if 
\begin{align}
	G_0(\omega) \,&=\, \gamma \qV \big(\omega^2 + (-\fu +\fv )\gamma\omega\big) 
	+\kappa_{D,0}(\omega) \big(\omega^2 + (-\fu +\fv )\gamma\omega\big)\notag\\
	&\qquad -\kappa_{D,0}(\omega) \big(\gamma \qv \omega-\gamma^2\big(\fu \qv -\fv \qu \big)\big) \label{eq:G0}
\end{align}
has no zeroes in $[0,\infty)$. Let us first consider the case $\omega>0$. For convenience we then rewrite
\begin{align*}
	\kappa_{D,0}(\omega) \,&=\,  D
        \left(\frac{ri_0'(r)}{i_0(r)}\right)\Big|_{r=\sqrt{\frac{\omega}{D}}}\,=\,
        \omega \tilde{\kappa}\left(\sqrt{\frac{\omega}{D}}\right),\\ 
	\tilde{\kappa}(r)\,&:=\, \frac{i_0'(r)}{ri_0(r)}.
\end{align*}
Explicit evaluation of $\tilde{\kappa}$ gives that $\tilde{\kappa}'<0$. In fact we have
\begin{align*}
	i_0(r)\,&=\, \frac{\sinh r}{r},\quad i_0'(r)\,=\, \frac{r\cosh r -\sinh r}{r^2},\\
	\tilde{\kappa}(r)\,&=\, \frac{r\cosh r-\sinh r}{r^2\sinh r},
\end{align*}
hence
\begin{align*}
	\tilde{\kappa}'(r)r^3\sinh^2 r \,&=\, -r^2 + 2\sinh^2r -r\sinh r\cosh r\\
	&=\, -r^2 -1 + \cosh 2r -\frac{r}{2}\sinh 2r\\
	&=\, \sum_{k\geq 2} \frac{1}{(2k)!}(2r)^{2k}(1-\frac{k}{2}) \,<\, 0.
\end{align*}
Furthermore we obtain
\begin{align}
	\lim_{r\to 0}\tilde{\kappa}(r) \,=\, \frac{1}{3},\quad \lim_{r\to\infty} \tilde{\kappa}(r) \,=\, 0. \label{eq:lim-tkappa-lim}
\end{align}
For $\omega> 0$ the equation $G_0(\omega)=0$ is  equivalent to
\begin{align}
	0\,&=\, \tilde{\kappa}\left(\sqrt{\frac{\omega}{D}}\right)\Big(\omega^2
          +\gamma\omega(\fv -\fu -\qv )+\gamma^2(\fu \qv -\fv \qu
          )\Big) + \gamma \qV \omega + \gamma^2\qV (\fv -\fu )\,=:\,
          \tilde{G}_0(\omega). \label{eq:case1-stab}
\end{align}
By \eqref{eq:lim-tkappa-lim} and $\qV >0$ we deduce that
$\lim_{\omega\to\infty} \tilde{G}_0(\omega)=+\infty$. Using
\eqref{eq:lim-tkappa-lim} we evaluate
\begin{equation*}
  \lim_{\omega \to 0} \tilde{G}_0(\omega) =
  \frac{1}{3}\gamma^2(\fu \qv -\fv \qu ) + \gamma^2\qV (\fv -\fu ),
\end{equation*}
and we obtain that 
\begin{align}
	0\,&\leq\, \frac{1}{3}(\fu \qv -\fv \qu ) + \qV (\fv -\fu ). \label{eq:case1-stab2=}
\end{align}
is a necessary condition for $\tilde{G}_0>0$ on $(0,\infty)$.

Let us next consider the case $\omega=l=0$ which gives $G_0(0)=0$. In
this case $u,v,V$ are all constant and satisfy, by \eqref{eq:u-lin},
\eqref{eq:v-lin} and \eqref{eq:mass-lin},  
\begin{align*}
	0\,&=\, f_u u + f_v v,\\
	0\,&=\, q_u u + q_v v + q_V V,\\
	0\,&=\, 4\pi( u+ v) + \frac{4\pi}{3}V.
\end{align*}
This system has a nontrivial solution if and only if $0\,=\, \frac{1}{3}(\fu \qv -\fv \qu ) + \qV (\fv -\fu )$. Together with \eqref{eq:case1-stab2=} this property proves that \eqref{eq:case1-stab2} is necessary for the asserted stability. 

From \eqref{eq:case1-stab2} we deduce that
\begin{align*}
	0\,&<\, (\fv -\fu )(\qV -\frac{1}{3}\qv ) +\frac{1}{3}(\qv -\qu )\fv  
	\,\leq\,(\fv -\fu )(\qV -\frac{1}{3}\qv ).
\end{align*}
By \eqref{cdt:q} this implies \eqref{eq:fu-fv-full} and in particular that \eqref{eq:fu-fv-full} is necessary for $\tilde{G}_0>0$ on $[0,\infty)$.

We claim that \eqref{eq:case1-stab2} is also sufficient to exclude a
nonnegative zero of $G_0$. Again, we consider for $\omega>0$
\begin{equation*}
	\tilde{G}_0(\omega)\,=\,
        \tilde{\kappa}\left(\sqrt{\frac{\omega}{D}}\right)\Big(\omega^2
        +\gamma\omega(\fv -\fu -\qv )+\gamma^2(\fu \qv -\fv \qu
          )\Big) + \gamma \qV \omega + \gamma^2\qV (\fv -\fu ).
\end{equation*}
As \eqref{eq:case1-stab2} implies \eqref{eq:fu-fv-full} we
have by \eqref{cdt:q-star} that $\fv -\fu -\qv >0$ and $\qV(\fv-\fu)>0$. If we now assume 
\begin{equation*}
  \fu \qv -\fv \qu > 0,
\end{equation*}
we immediately conclude $\tilde{G}_0(\omega) > 0$ for all $\omega>0$. If on the other hand
\begin{equation*}
  \fu \qv -\fv \qu \le 0,
\end{equation*}
we remark that $\tilde{\kappa}$ is decreasing, hence $\tilde{\kappa}
\le \frac1{3}$ on $[0,\infty)$ and therefore
\begin{equation*}
  \tilde{G}_0(\omega) > \, \gamma^2\big(\frac{1}{3}(\fu \qv -\fv \qu
  )+\qV (\fv -\fu )\big) \,\geq\, 0
\end{equation*}
for all $\omega>0$, which proves the claim. We have already seen above that $0\neq \frac{1}{3}(\fu \qv -\fv \qu ) + \qV (\fv -\fu )$ is sufficient to exclude that constant perturbations $(u,v,V)$, corresponding to the case $\omega=0$, are solution of the linearized system.

It only remains to prove that \eqref{eq:case1-stab2} is sufficient to exclude an instability for $l=0$ in the case $V_{00}=0$. However, in this case by \eqref{eq:cdt-Vlm0},  \eqref{eq:om-Vlm0} we have $\omega=0$ and we are in the case that $(u,v,V)$ is constant, which is excluded by \eqref{eq:case1-stab2}, as we have seen above.

\end{proof}
\subsection{Instability conditions}
We next characterize instabilities of our system in a spatially homogeneous stationary point $(u_*,v_*,V_*)$ as above under general perturbations.
\begin{theorem}\label{thm:instab}
Assume that \eqref{eq:case1-stab2} is satisfied and that
\begin{align}
  0\,&> (\gamma \qV +Dl) \big(dl^2(l+1)^2 +\gamma l(l+1)(-d\fu +\fv )\big) \notag\\
  &\qquad -D \gamma \qv l^2(l+1) + Dl\gamma^2 (\fu \qv -\fv \qu ) \label{eq:ca1-instab}
\end{align}
holds in $(u_*,v_*,V_*)$. Then the system  \eqref{eq:V}--\eqref{eq:flux1} is linearly asymptotically unstable in $(u_*,v_*,V_*)$.

If \eqref{eq:case1-stab2} and $d=1$
hold 
then \eqref{eq:ca1-instab} is also necessary for an instability.
\end{theorem}
\begin{proof}
We again first restrict to the case $V_{lm}\neq 0$. Note that the right-hand side of \eqref{eq:ca1-instab} coincides with $G_l(0)$ as defined in \eqref{eq:eigenvalue-eq}. 
In order to show the existence of an instability, we prove that there
is a positive zero $\omega_{lm} > 0$ of $G_l$. The modified Bessel
function of the first kind $i_l$ have by \cite[10.52.1,
10.52.5]{DLMF}, \cite{OLBC10} the asymptotic expansions
\begin{align*}
	i_l(r)\,\approx\, \frac{1}{2r}e^r\quad\text{ as }r\to\infty,\qquad
	i_l(r)\,\approx\, \frac{1}{(2l+1)!}r^l\quad\text{ as }r\to 0.
\end{align*}
This implies 
\begin{align}
	\kappa_{D,l}(\omega) =
        D\left(\frac{ri_l'(r)}{i_l(r)}\right)\Big|_{r=\sqrt{\frac{\omega}{D}}}
        \,&\to\, D l \text{ as }\omega\to 0, \label{eq:kappa0} 
\end{align}
and $\kappa_{D,l}(\omega)\to \infty$ as $\omega\to\infty$ hence
\begin{align*}
	\lim_{\omega \to \infty} G_l(\omega) \,&=\, +\infty.
\end{align*}
We therefore obtain that \eqref{eq:ca1-instab} is sufficient to guarantee a solution $\omega>0$ of $G_l(\omega)=0$. 
It remains to prove that \eqref{eq:ca1-instab} is also necessary if $d=1$ holds. 
We will need some information on the derivative of $\kappa_{D,l}$. We start by observing that by \eqref{eq:2.24}
\begin{align*}
	\kappa_{D,l}'(\omega)\,=\, \Big[\frac{1}{2r}\Big(\frac{r i_l'(r)}{i_l(r)}\Big)'\Big]_{r=\sqrt{\frac{\omega}{D}}}.
\end{align*}
By the definition of $i_l$ we compute
\begin{align*}
	\frac{r i_l'(r)}{i_l(r)} \,=\, -\frac{1}{2} + \frac{rI_{l+1/2}'(r)}{I_{l+1/2}(r)}.
\end{align*}
By \cite{Gr32} however we know that the quotient on the right-hand
side has strictly positive derivative on $\R^+$. This implies that
also $\kappa_{D,l}'>0$. Moreover, from \cite{IfSi90,Ko12} we obtain
\begin{equation*}
  \frac{r i_l'(r)}{i_l(r)} \le l + \frac1{3}r^2
\end{equation*}
which yields
\begin{equation}
  \label{eq:kappaEst}
  \kappa_{D,l}(\omega) \le lD + \frac1{3}\omega = \kappa_{D,l}(0) +
  \frac1{3}\omega 
\end{equation}
for all $\omega>0$.

This information implies by \eqref{eq:eigenvalue-eq},
\eqref{eq:fu-fv-full}, and $d=1$ that 
\begin{align}
  G_l(\omega)\,&\geq\, \gamma \qV \Big(dl^2(l+1)^2 +\gamma
  l(l+1)(-d\fu +\fv ) + (-\fu + \fv) \gamma \omega \Big) \notag\\
  &\quad +\kappa_{D,l}(\omega) \Big(dl^2(l+1)^2 +\gamma l(l+1)(-d\fu +\fv
  )\Big)\notag\\
  &\quad +\kappa_{D,l}(\omega) \Big(-\gamma \qv
  \big(l(l+1)+\omega_{lm}\big)+\gamma^2\big(\fu \qv -\fv \qu
  \big)\Big)\notag\\
  &\geq G_l(0) + \gamma^2 \qV (\fv - \fu) \omega
  + \gamma^2\big(\kappa_{D,l}(\omega)- \kappa_{D,l}(0)\big)(\fu\qv - \fv \qu). \label{eq:2.39bis}
\end{align}
In the case $(\fu\qv - \fv \qu)>0$ this yields $ G_l(\omega) \ge  G_l(0)$ for all $\omega>0$, as $\kappa_{D,l}$ is increasing (see above). In the case $(\fu\qv - \fv \qu) \le 0$ we obtain from \eqref{eq:2.39bis} that
\begin{align*}
  G_l(\omega) & \hspace{-2.2mm}\stackrel{\eqref{eq:case1-stab2}}{\geq} G_l(0)
  -\gamma^2(\fu\qv - \fv \qu)\Big(\frac1{3}\omega - \kappa_{D,l}(\omega) +
  \kappa_{D,l}(0)\Big) \stackrel{\eqref{eq:kappaEst}}{\geq} G_l(0).
\end{align*}
for all $\omega>0$.
From this we conclude that $G_l(0)<0$ is also necessary for the
existence of an instability with $V_{lm}\neq 0$. 

It remains to consider the possibility that an instability with $V_{lm}=0$ exists. In this case we obtain from \eqref{eq:cdt-Vlm0} and $d=1$ that $-f_uq_v+f_vq_u=0$. But then \eqref{eq:om-Vlm0} implies that $\omega\leq 0$.
\end{proof}
\begin{corollary}\label{cor:instab}
Assume \eqref{eq:case1-stab2}. Then the instability condition \eqref{eq:ca1-instab} holds if the conditions of Case 1 or Case 2 below are satisfied and if $D>0$ is chosen sufficiently large.
\begin{itemize}[labelindent=*,leftmargin=*]
  \item
    \text{\rm Case 1:}
    \begin{align}
      f_u q_v - f_vq_u \,\geq\, 0, \label{eq:Ca1-1}\\
      df_u -f_v + q_v \,> 0, \label{eq:Ca1-2}\\
      Q\,:=\, (df_u -f_v + q_v)^2 - 4d(f_u q_v-f_v q_u) \,> 0, \label{eq:Ca1-3}
    \end{align}
    and there exists an $l\in\N$ with
    \begin{align}
      \lambda_- \,<\, \frac{l(l+1)}{\gamma} \,<\, \lambda_+, \label{eq:Ca1-4}
    \end{align}
    where 
    \begin{align}
      \lambda_{\pm} \,=\, \frac{1}{2d}\Big( df_u -f_v + q_v \pm \sqrt{Q}\Big). \label{eq:Ca1-5}
    \end{align}
  \item
    \text{\rm Case 2:}
    \begin{align}
      f_u q_v - f_vq_u \,<\, 0 \label{eq:Ca2-1}
    \end{align}
    and there exists an $l\in\N$ with
    \begin{align}
      \frac{l(l+1)}{\gamma} \,<\, \lambda_+, \label{eq:Ca2-2}
    \end{align}
    where $\lambda_+$ is as defined in \eqref{eq:Ca1-5}.
  \end{itemize}
\end{corollary}
\begin{proof}
In order to evaluate \eqref{eq:ca1-instab} for $D\gg 1$ we consider the coefficient of $Dl$, that is
\begin{align}
    e\,:=\, dl^2(l+1)^2 +  \gamma l(l+1)(- df_u + f_v - q_v)
    + \gamma^2 (f_u q_v - f_v q_u). \label{eq:corr-2}
\end{align}
In Case 1 the last term on the left-hand side is by \eqref{eq:Ca1-1} nonnegative and $e<0$ holds if and only if the conditions \eqref{eq:Ca1-2}--\eqref{eq:Ca1-4} are satisfied.\\
  In Case 2 the last term on the left-hand side of \eqref{eq:corr-2} is negative  and $e<0$ holds if and only if the condition \eqref{eq:Ca2-2} is satisfied.
Therefore $e<0$ holds if and only if Case 1 or Case 2 are satisfied.
We now observe that the term $Dle$ becomes dominant in \eqref{eq:ca1-instab} for $D\gg 1$ and we deduce from Theorem \ref{thm:instab} that if $e<0$ then for $D$ sufficiently large an instability exists.
\end{proof}
\begin{remark} (1) For $d=1$ we deduce from Theorem \ref{thm:instab} and \eqref{eq:case1-stab2}, \eqref{eq:ca1-instab} that a diffusive instability exists if and only if
\begin{align*}
  0\,&> (\gamma \qV +Dl) \big(l^2(l+1)^2 +\gamma l(l+1)(-\fu +\fv )\big) 
   -D \gamma \qv l^2(l+1) + Dl\gamma^2 (\fu \qv -\fv \qu )
\end{align*}
holds. By \eqref{eq:fu-fv-full} we deduce that \eqref{eq:Ca2-1} is necessary and that only Case 2 is a possible scenario for an instability. By Corollary \ref{cor:instab} this condition and \eqref{eq:Ca2-2} for an $l\in\N$ are also sufficient to ensure, for $D$ sufficiently large, an instability. In particular, for $d=1$ there exist parameter values such that the system has a diffusive instability.
\medskip

\noindent (2) Assume \eqref{eq:case1-stab2}. Then we observe from \eqref{eq:eigenvalue-eq} that perturbations in directions of eigenvectors $\varphi_{lm}$ decay for all sufficiently large $l\in\N$.
\medskip

\noindent (3) Case 1 or Case 2 in Corollary \ref{cor:instab} are sufficient but not necessary for an instability. A third case may arise for $d\gg 1$ and $D$ sufficiently small. In fact, even if the factor that multiplies $\kappa_{D,l}(\omega)$ in \eqref{eq:eigenvalue-eq} is positive, this term might be dominated by the first line in \eqref{eq:eigenvalue-eq}, which becomes negative if $df_u-f_v \gg 1 $. As we are mostly interested in $d=1$ we do not investigate this case further.
\end{remark}
\begin{remark}\label{rem:stab-ode}
We finally would like to relate the distinction between Case 1 and Case 2 instabilities, given by the inequalities \eqref{eq:Ca1-1} and \eqref{eq:Ca2-1}, to the stability properties of the zero lateral diffusion reduction of the full system \eqref{eq:V}--\eqref{eq:flux1}. This reduction is given by choosing $d_u=d_v=0$ in the dimensional formulation of our system, see \eqref{eq:f-app} and \eqref{eq:q-app} in the appendix and leads to the system
\begin{align}
  \label{eq:Vode}
  \pd_t  V &=  D \Delta V \quad &&\text{in} \quad
  B  \times  I,\\
  \label{eq:uode}
  \pd_t  u &=  \gamma f(u,v)\quad &&\text{on} \quad
  \Gamma \times  I,\\
  \label{eq:vode}
  \pd_t  v &=  \gamma (- f(u,v) + q(u,v,V)) \quad &&\text{on}
  \quad  \Gamma \times  I,\\  
  -D \nabla V \cdot \nu &= \gamma q(u,v,V)  \quad &&\text{on} \quad \Gamma\times I.\label{eq:flux1ode}
\end{align}
An instability of the corresponding system is then characterized by the existence of a positive root $\omega_{lm}$ of
\begin{align}
	0 \,&=\, \gamma \qV \big(\omega_{lm}^2 + (-\fu +\fv )\gamma\omega_{lm}\big) \notag\\
	&\qquad +\kappa_{D,l}(\omega_{lm}) \Big(\omega_{lm}^2 +(-\fu +\fv-\qv )\gamma\omega_{lm}+\gamma^2\big(\fu \qv -\fv \qu \big)\Big)\label{eq:eigenvalue-eqode}
\end{align}
with $\kappa_{D,l}$ as in \eqref{eq:2.24}. We therefore see that under
the stability assumption \eqref{eq:case1-stab2} in Case 1, i.e. if
\eqref{eq:Ca1-1} holds, the system
\eqref{eq:Vode}--\eqref{eq:flux1ode} is stable, whereas for Case 2,
i.e. if \eqref{eq:Ca2-1} holds, the system is unstable for $D\gg 1$
and $\gamma$ chosen large enough. This shows that the second
instability mechanism is not induced by the membrane diffusion but
rather by the cytosolic diffusion. See Section \ref{sec:discuss} for a further
discussion.
\end{remark}
\section{Stability Analysis for the non-local reduction $D\to\infty$}\label{sec:turingNonLocalA}
By formally letting $D \to \infty$ in \eqref{eq:V}--\eqref{eq:flux1} one obtains the
following reduced two-variable system
\begin{align}
  \partial_t u \,&=\, \Delta_\Gamma u + \gamma f(u,v), \label{eq:u-gen}\\
  \partial_t v \,&=\, d\Delta_\Gamma v + \gamma \left(-f(u,v) +
    q(u,v,V[u+v])\right), \label{eq:v-gen}
\end{align}
where $V[u+v]$ is the \emph{non-local} functional
\begin{gather}
  V[u+v]\,=\, V_{init} - c \int_{\Gamma} (u + v) \dg, \label{eq:Vnew}
\end{gather}
with $V_{init}>0$ given and $c:=\frac1{|B|}$. Note that $V_{init}$ is determined
by the total mass of GTPase, which is constant in time. The system \eqref{eq:u-gen}--\eqref{eq:Vnew} has already been considered in \cite{RaRo12} and has, compared to the fully coupled system, the advantage of having one fixed domain of definition (the membrane $\Gamma$). The remnant of the spatial coupling in the full system is the non-locality, introduced by the specific form of $V=V[u+v]$. In \cite{RaRo12} we have, among other things, presented a stability analysis, which however was not complete in the characterization of instabilities \cite{RaRo13}. Here we complete that discussion and obtain a characterization that coincides with the behavior of \eqref{eq:V}--\eqref{eq:flux1} for large cytosolic diffusion constant $D$. Moreover we obtain some additional properties of instabilities that are more difficult to characterize for the fully coupled system.

In the following stability analysis, in contrast to the one of the full system, we do not need to restrict ourselves to spherical cell shapes. We therefore fix an arbitrary open, bounded domain $B\subset\R^3$ with smooth connected boundary $\Gamma=\partial B$. We assume again that $f$ and $q$ satisfy \eqref{cdt:q}, consider a spatially homogeneous stationary point $(u_*,v_*)$ of \eqref{eq:u-gen}--\eqref{eq:Vnew}, and set $V_*:=V[u_*,v_*]$. Then $(u_*,v_*)$ is also a stationary point of the ODE reduction of \eqref{eq:u-gen}, \eqref{eq:Vnew},
\begin{align}
  \partial_t u \,&=\, \gamma f(u,v), \label{eq:u-gen-ode}\\
  \partial_t v \,&=\, \gamma \left(-f(u,v) + q(u,v,V_1(u+v))\right), \label{eq:v-gen-ode}
\end{align}
where
\begin{gather*}
  V_1(u+v)\,=\,V_{init} -c|\Gamma|(u+v).
\end{gather*}
Note that $V_1$ is just a (non-local) real function and that $V_1'=-c|\Gamma|<0$. Again it is convenient to introduce the notation
\begin{align*}
	f_u\,&:=\, \partial_u f(u_*,v_*),\quad f_v\,:=\, \partial_v f(u_*,v_*),\\
	q_u\,&:=\,\partial_u q(u_*,v_*,V_*)\quad q_u\,:=\,\partial_v q(u_*,v_*,V_*),\quad
	q_V\,:=\,\partial_V q(u_*,v_*,V_*).
\end{align*}

The stability of the ODE system  \eqref{eq:u-gen-ode}, \eqref{eq:v-gen-ode} in $(u_*,v_*)$ is equivalent to the conditions
\begin{align}
  0\,&>\, f_u -f_v  + q_v + q_VV_1', \label{eq:Tu-gen-1}\\
  0\,&<\, f_u(q_v+q_VV_1') - f_v (q_u+q_VV_1'). \label{eq:Tu-gen-2}
\end{align}
This also corresponds to the stability of \eqref{eq:u-gen},
\eqref{eq:v-gen} in $(u_*,v_*)$ with respect to spatially homogeneous
perturbations. We remark that \eqref{cdt:q} and \eqref{eq:Tu-gen-2} imply that
\begin{align*}
	0\,&<\, f_uq_v -f_vq_u +q_V V_1'(f_u-f_v)\,=\, (f_u-f_v)(q_v+q_V V_1') + f_v(q_v-q_u)\\
	&\leq\, (f_u-f_v)(q_v+q_V V_1'),
\end{align*}
which by \eqref{cdt:q} yields that
\begin{gather}
	f_u\,<\, f_v. \label{eq:fu-fv}
\end{gather}
In particular we see that under the assumption \eqref{cdt:q} the inequality \eqref{eq:Tu-gen-2} already implies \eqref{eq:Tu-gen-1}.

We remark that \eqref{eq:Tu-gen-2} coincides in the case of a spherical cell, i.e. $\Gamma=S^2\subset\R^3$ with the stability condition \eqref{eq:case1-stab2} for $D<\infty$. In fact, in this case we have $c|\Gamma| = \frac{4\pi}{4\pi/3}\,=\, 3$ and we obtain for the right-hand side in \eqref{eq:Tu-gen-2} that
\begin{align*}
	f_u(q_v+q_VV_1') - f_v (q_u+q_VV_1')\,=\, f_uq_v-f_vq_u  - 3 q_V( f_u -f_v)
\end{align*}
and the equivalence of \eqref{eq:Tu-gen-2} and \eqref{eq:case1-stab2} follows.

For the instability of \eqref{eq:u-gen}--\eqref{eq:v-gen} in $(u_*,v_*)$ we obtain the following characterization.
\begin{proposition}\label{prop:inst-new}
   Assume that conditions \eqref{cdt:q} and \eqref{eq:Tu-gen-1},
  \eqref{eq:Tu-gen-2} hold. Then the system \eqref{eq:u-gen},
  \eqref{eq:v-gen} is unstable in $(u_*,v_*)$ if and only if in this point either the conditions from Case 1 or  Case 2 below are satisfied:
\begin{itemize}[labelindent=*,leftmargin=*]
  \item
    \text{\rm Case 1:}
    \begin{align}
      f_u q_v - f_vq_u \,\geq\, 0,\label{eq:Ca1-1-Dinfty}\\
      df_u -f_v + q_v \,> 0, \label{eq:Ca1-2-Dinfty}\\
      Q\,:=\, (df_u -f_v + q_v)^2 - 4d(f_u q_v-f_v q_u) \,> 0, \notag
    \end{align}
    and there exists an eigenvalue $\mu>0$ of $-\Delta_{\Gamma}$ with
    \begin{align*}
      \lambda_- \,<\, \frac{\mu}{\gamma} \,<\, \lambda_+,
    \end{align*}
    where 
    \begin{align}
      \lambda_{\pm} \,=\, \frac{1}{2d}\Big( df_u -f_v + q_v \pm \sqrt{Q}\Big). \label{eq:Ca1-5-Dinfty}
    \end{align}
  \item
    \text{\rm Case 2:}
    \begin{align}
      f_u q_v - f_vq_u \,<\, 0 \label{eq:Ca2-1-Dinfty}
    \end{align}
    and there exists an eigenvalue $\mu>0$ of $-\Delta_{\Gamma}$ with
    \begin{align}
      \frac{\mu}{\gamma} \,<\, \lambda_+, \label{eq:Ca2-2-Dinfty}
    \end{align}
    where $\lambda_+$ is as defined in \eqref{eq:Ca1-5-Dinfty}.
  \end{itemize}
\end{proposition}
\begin{proof}
The linearization of \eqref{eq:u-gen}, \eqref{eq:v-gen} in $(u_*,v_*)$ is given by
\begin{align}
  \partial_t u \,&=\, \Delta_\Gamma u + \gamma f_u u +\gamma f_v v, \label{eq:lin-u-gen}\\
  \partial_t v \,&=\, d\Delta_\Gamma v + \gamma \left(-f_uu-f_v v +
    q_u u +q_v v -c q_V\int_{\Gamma} (u + v) \dgx\right). \label{eq:lin-v-gen}
\end{align}
It suffices to consider perturbations of the form
\begin{gather}
	u(y,t) \,=\,  a e^{\omega t} \psi(y)  ,\qquad v(y,t) \,=\,  b e^{\omega t} \psi(y) \label{eq:ansatz-lin}
\end{gather}
with $a,b\in\R$, where $\psi$ is an eigenvector of $-\Delta_\Gamma$ to an eigenvalue $\mu$.
The operator $-\Delta_\Gamma$ has only countably many eigenvalues that are nonnegative. Zero is a simple eigenvalue with eigenspace given by the constant functions on $\Gamma$. As we have considered spatially homogeneous perturbations already above we can restrict ourselves to $\mu>0$. Any eigenvector for an eigenvalue $\mu>0$ satisfies
\begin{align*}
	\int_{\Gamma} (u + v) \dgx\,=\, 0.
\end{align*}
Then $(u,v)$ as in \eqref{eq:ansatz-lin} is a solution of \eqref{eq:lin-u-gen}, \eqref{eq:lin-v-gen} if and only if
\begin{align}
   0 &= \omega^2 + \omega\big((d+1)\mu + \gamma (- f_u +f_v -q_v)\big) \notag\\
   &\quad + d\mu^2 +  \gamma \mu(- df_u + f_v - q_v)
   + \gamma^2 (f_u q_v - f_v q_u). \label{eq:corr-1-Dinfty}
\end{align}
The inequality \eqref{eq:fu-fv} implies that in \eqref{eq:corr-1-Dinfty} the term on the right-hand side that is linear in $\omega$ is positive for positive $\omega$. A positive zero of this equation therefore exists if and only if
\begin{gather}
    d\mu^2 +  \gamma \mu(- df_u + f_v - q_v)
    + \gamma^2 (f_u q_v - f_v q_u) \,<\, 0. \label{eq:corr-2-Dinfty}
\end{gather}
which is identical to \eqref{eq:corr-2} for $\mu=l(l+1)$. In Corollary \ref{cor:instab} we have proved that this condition is equivalent to the property that Case 1 or Case 2 hold.
\end{proof}
By Proposition \ref{prop:inst-new} the (sufficient) instability conditions from Corollary \ref{cor:instab} are sharp for $D=\infty$. We remark that in the classical local case, which corresponds to $q_V =0$, by \eqref{eq:Tu-gen-2} only Case 1 is possible, which just describes the usual conditions for an Turing instability. Case 2 on the other hand represents a different mechanism that is not present for local two-variable systems.

Similarly as in Remark \ref{rem:stab-ode} we observe for the non-local system \eqref{eq:u-gen}--\eqref{eq:Vnew} that the inequality \eqref{eq:Ca1-1-Dinfty} that characterizes Case 1 corresponds to the stability of the non-local ODE system \eqref{eq:u-gen-ode}, \eqref{eq:v-gen-ode} with respect to spatially heterogeneous perturbations. Due to the non-locality this property however is not equivalent to the stability of the non-local reaction--diffusion system with respect to spatially homogeneous perturbations. 
In particular, even for zero lateral diffusion in the case that \eqref{eq:Ca2-1-Dinfty}, \eqref{eq:Ca2-2-Dinfty} hold the non-local system in unstable with respect to spatially heterogeneous perturbations. See Section \ref{sec:discuss} for a further discussion.

In Case 1 we deduce from \eqref{eq:Ca1-2-Dinfty} that $f_u>0$ and further, by \eqref{cdt:q} and \eqref{eq:Ca1-1-Dinfty}
\begin{align*}
	0\,&\leq\, f_uq_v -f_vq_u \,\leq\, (f_u-f_v)q_u,
\end{align*}
hence $q_u \leq 0$. In particular, in Case 1 the stationary point $(u_*,v_*)$ needs to be of activator--substrate-depletion type. In contrast Case 2 is less restrictive, and does allow for stationary points with $f_u\leq 0$ and $q_u\geq 0$. 

We further observe that for equal lateral diffusion $d=1$ no instabilities of  \eqref{eq:u-gen}, \eqref{eq:v-gen} exist in Case 1. In fact \eqref{eq:fu-fv}, \eqref{eq:Ca1-2-Dinfty}, and \eqref{cdt:q} would imply
\begin{gather*}
    0\,>\, f_u -f_v  \,> \,-q_v \,\geq\, 0,
\end{gather*}
which gives a contradiction. In contrast, in Case 2, i.e. under the condition \eqref{eq:Ca2-1-Dinfty}, for any $d \geq 0$ there exists $\gamma>0$ such that \eqref{eq:Ca2-2-Dinfty} is satisfied and an instability exists.

A particular property, different from the classical Turing instability is that for $d=1$ in Case 2 the most unstable perturbations of system
\eqref{eq:u-gen}, \eqref{eq:v-gen} is always in direction of an eigenvector corresponding to the smallest
 positive eigenvalue $\mu$. In fact, if we consider the unique positive root $\omega=\omega(\mu)$ of \eqref{eq:corr-1-Dinfty} as function of $\mu>0$ we observe that $\omega(\mu) + \mu$ is independent of $\mu$ and we therefore deduce that $\omega(\mu)$ is decreasing in $\mu$.
\section{Numerical Treatment of the Full
  System}\label{sec:numericsFull}

In the following, we present numerical simulations of 
\eqref{eq:V}--\eqref{eq:flux1}. These confirm the results
of the linear stability analysis of Section \ref{sec:turingFullA} and
in addition allow to study the behavior beyond the linear regime.

\subsection{Phase-field approach for coupling bulk- and surface PDE's}

In order to numerically treat equations on the membrane and inside the cell, we
use a phase-field approach. A diffuse-interface description of coupled bulk
diffusion and ordinary differential equations on the bounding surface
has been proposed in \cite{LeRa05} to simulate membrane-bound Turing
patterns. In \cite{RaVo06} a diffuse-interface approach for solving
PDE's on surfaces has been introduced. Moreover, in \cite{LiLoRaVo09}
a diffuse-interface approximation for PDE's in domains with boundaries
implicitly given by phase field functions has been provided, see also
\cite{KoLeRa03} for the special case of no-flux boundary
conditions. In order to treat the spatially coupled system \eqref{eq:V}--\eqref{eq:flux1} we combine both methods. For a similar approach see
\cite{TeLiLoWaVo09}. Alternative methods, different from a phase-field approach have also been used in similar contexts. A finite element analysis for a
coupled bulk--surface equation has recently been presented in \cite{ElRa13}. In \cite{NoGaChReScSl07} finite
volume techniques are applied to reaction--diffusion
equations on curved surfaces, coupled to diffusion in the volume.  

We use here the diffuse-interface approach as a convenient numerical 
method. It can more easily be adapted to complicated domains and realistic cell shapes. In this case the main effort is to construct a suitable 
discrete signed-distance function from the cell boundary, which is often easier to obtain than a triangulation of the boundary, necessary in other methods. Furthermore, coupling of equations in the bulk and on its boundary does not require any coupling of meshes
with different dimensions. Finally, an extension of the phase-field approach to evolving membrane shapes is in principle relatively easy (though costly) and even allows to include topological changes. On the other hand, solving partial differential equations on $\Gamma$ is computationally certainly more expensive in a diffuse-interface setting.

The strategy of the phase-field approach is as follows: We choose a (simple) computational domain $\Omega$
containing $\overbar{B}$ and we introduce a smeared-out indicator
function $\phi: \Omega \co \R$, for example given by
\begin{equation*}
  \phi(x) := \frac{1}{2}(1 - \tanh(3 r(x)/\eps)),
\end{equation*}
where $r$ denotes the signed distance from $\Gamma$, chosen negative
inside $B$ and positive outside $\overbar{B}$. The surface $\Gamma$
is then given by the level set $\{\phi = \frac1{2}\}$. The
corresponding `diffuse interface' is understood as the layer where
$\phi$ is away from $\pm 1$. The order of the diffuse interface width
is then determined by the (small) parameter $\eps>0$.
  
We define $b(z) := 36 z^2 (z - 1)^2$ for $z \in \R$. According to
\cite{RaVo06} and  \cite{LiLoRaVo09}, a
diffuse-interface approximation for the coupled
system \eqref{eq:V}--\eqref{eq:flux1} is given by 
\begin{align}
  \label{eq:diff1}
  \phi \pd_t V &= D \D \cdot (\phi \D V) -\eps^{-1}
  b(\phi)\gamma q(u,v,V),\\
  \label{eq:diff2}
  b(\phi) \pd_t u &= \D \cdot
  (b(\phi) \D u) + b(\phi) \gamma f(u,v),\\
  \label{eq:diff3}
  b(\phi) \pd_t v &=  d \D \cdot
  (b(\phi) \D v)+b(\phi)\gamma(- f(u,v) + q(u,v,V))
\end{align}
for unknown functions $u, v, V: \Omega \times I \to \R$. We complement this system by initial conditions 
\begin{align*}
  &V(\cdot, 0) \,=\, V_0,\quad v(\cdot, 0) \,=\, v_0,\quad
  u(\cdot, 0) \,=\, u_0
\end{align*}
for given extensions to $\Omega$ of the original initial conditions $u_0, v_0, V_0$, which were only defined on $\Gamma$ and $B$, respectively. In
\eqref{eq:diff1} the phase field function $\phi$ restricts the time
derivative and diffusion to the cell, while the function $b(\phi)$
restricts the flux $q$ to the membrane. Accordingly, in
\eqref{eq:diff2}, \eqref{eq:diff3} the function  $b(\phi)$ is applied to
restrict the reaction diffusion equations to the membrane.

\subsection{Numerical Approach}

We consider either the sphere $\Gamma = S^2$ or an ellipsoid $\Gamma$
with semi-axes $a, b$ and $c$ in $x_1$-, $x_2$- and $x_3$-direction,
respectively. To discretize in time we use a
semi-implicit Euler scheme with all nonlinearities linearized
corresponding to a single Newton step. We choose a computational
domain $\Omega := (-2,2)^3$ containing $\overbar{B}$, and we use an
adaptively refined mesh in order to discretize system
\eqref{eq:diff1}--\eqref{eq:diff3} in space using linear finite
elements. On the boundary $\pd \Omega$ we assume periodicity of the
discrete solutions $u_h, v_h, V_h$. Due to the degeneracy of equations
\eqref{eq:diff1}--\eqref{eq:diff3} and to avoid numerical problems we
regularize \eqref{eq:diff1}--\eqref{eq:diff3} in all second order
terms by adding a small positive number $\delta$ to $b(\phi)$. The
resulting linear system of equations is solved by a stabilized
bi-conjugate gradient method (BiCGStab) for $(u_h, v_h, V_h)$ in each
time step. The resulting scheme has been implemented in the adaptive
FEM toolbox AMDiS \cite{VeVo07}.

\subsection{Numerical Examples}

In all computations, we use $f, q$ as given in \eqref{eq:f-app} and
\eqref{eq:q-app}, respectively. We assume random initial
conditions $u_0: \Omega \to [0,0.0002]$, $v_0:
\Omega \to [0,0.0002]$. Moreover, we choose a constant initial
condition $V_0$ for $V_h$ such that the expected value of the total mass in the system is given by $V_{init}|B|$ and $V_{init}=5.1$, which is the value used in Section \ref{sec:numericsNonLocal} for the reduced system. This choice results in the case of a spherical cell in the initial condition $V_0 =  5.0994$ for the cytosolic concentration $V_h$.

For the parameters that determine $f$ and $q$ in \eqref{eq:f-app}, \eqref{eq:q-app} we chose the values given in Table \ref{table}. In particular we
always assume $d=1$ corresponding to equal lateral diffusion constants
for $u$ and $v$. Note that for this choice, Case 2 in Corollary
\ref{cor:instab} applies and guarantees an instability for $D > 0$
sufficiently large.

\begin{table}[h]
  \begin{center}
    \begin{tabular}[h]{| c | c | c | c | c | c | c |
        c | c | c | c | c | c | c | c | c | c | c | c |}
      \hline
      parameter & $d$ & $\gamma$ & $a_1$ &
      $a_2$ & $a_3$ 
      & $a_4$ & $a_5$
      & $a_6$  & $a_{-6}$ & $\eps$ & $\delta$ \\
      \hline
      value & $1$ & $400$ & $0.02$ & $20$ & $160$ &
      $1$ & $0.5$
      & $0.36$ & $5$ & $0.1$ & $10^{-6}$ \\
      \hline
    \end{tabular}
    \vskip5pt\caption{\footnotesize Parameters used for numerical
      results.}\label{table}
  \end{center}
\end{table}

\subsubsection{Instability for large cytosolic diffusion
  coefficient}

First, we choose $D = 100$. In Fig. \ref{fig:D=100} we see results in
this case, showing contour plots of the solutions $u_h, v_h, V_h$
evaluated on the level set $\Gamma_h := \{\phi_h = 1/2\}$ at different
times. Thereby, one observes the evolution to an unstable stationary
solution and towards an equilibrium with local maxima of $u_h$ and
$v_h$ on $\Gamma_h$. In a way this result shows similarities to
Turing--type instabilities, where usually differences in diffusion
constants drive the instability. In this case, $d = 1$ corresponds to
equal lateral diffusion constants. However, the large cytosolic
diffusion admits the development of heterogeneities.

\begin{figure}[here]
  \centerline{
    \hfill
    \subfigure[$u_h(t = 0)$]{
      \includegraphics*[width=0.14\textwidth]{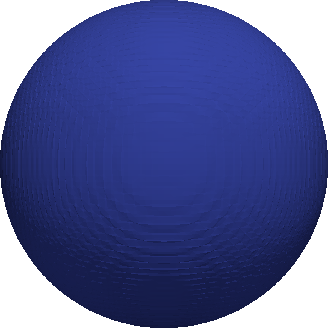}
    }
    \hfill
    \subfigure[$u_h(t = 0.5235)$]{
      \includegraphics*[width=0.14\textwidth]{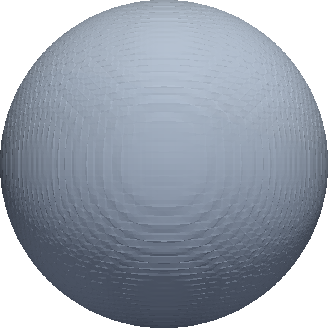}
    }
    \hfill
    \subfigure[$u_h(t = 1.0235)$]{
      \includegraphics*[width=0.14\textwidth]{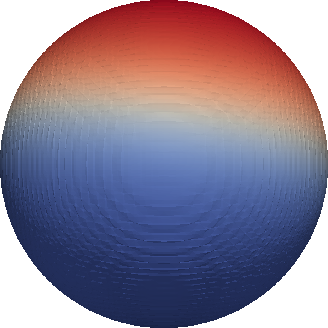}
    }
    \hfill
    \subfigure[$u_h(t = 5.0235)$]{
      \includegraphics*[width=0.14\textwidth]{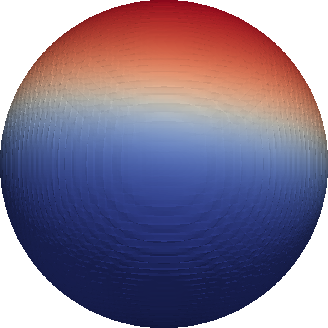}
    }
    \hfill
    \subfigure[]{
      \includegraphics*[width=0.05\textwidth]{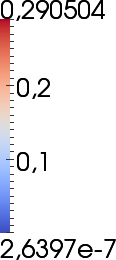}
    }
    \hfill
  }
  \vspace{3ex}
  \centerline{
    \hfill
    \subfigure[$v_h(t = 0)$]{
      \includegraphics*[width=0.14\textwidth]{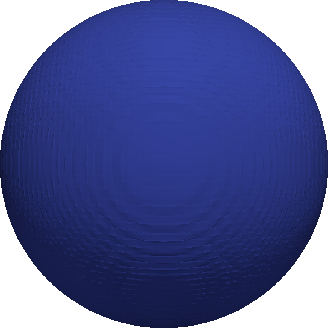}
    }
    \hfill
    \subfigure[$v_h(t = 0.5235)$]{
      \includegraphics*[width=0.14\textwidth]{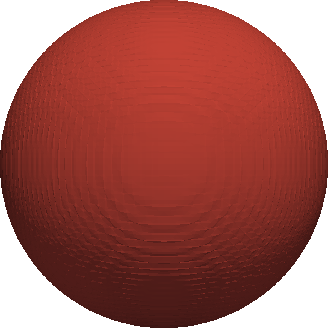}
    }
    \hfill
    \subfigure[$v_h(t = 1.0235)$]{
      \includegraphics*[width=0.14\textwidth]{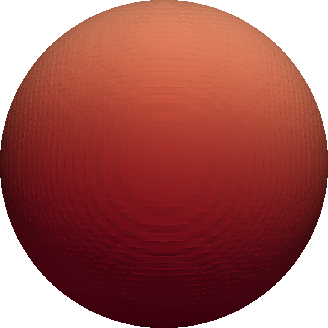}
    }
    \hfill
    \subfigure[$v_h(t = 5.0235)$]{
      \includegraphics*[width=0.14\textwidth]{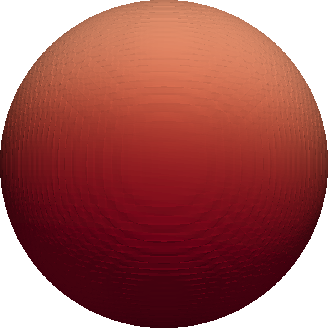}
    }
    \hfill
    \subfigure[]{
      \includegraphics*[width=0.05\textwidth]{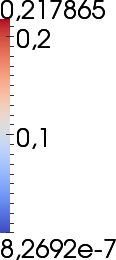}
    }
    \hfill
  }
  \vspace{3ex}
  \centerline{
    \hfill
    \subfigure[$V_h(t = 0)$]{
      \includegraphics*[width=0.14\textwidth]{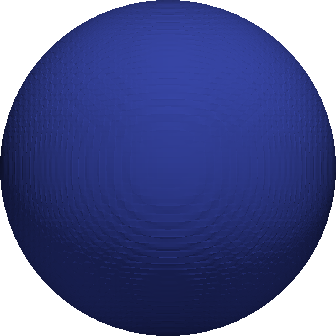}
    }
    \subfigure[]{
      \includegraphics*[width=0.0365\textwidth]{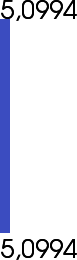}
    }
    \hfill
    \subfigure[$V_h(t = 5.0235)$]{
      \includegraphics*[width=0.14\textwidth]{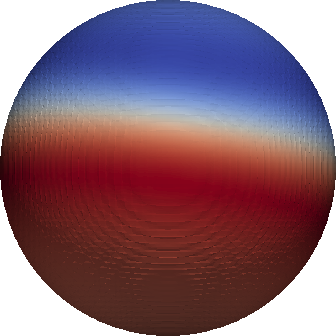}
    }
    \subfigure[]{
      \includegraphics*[width=0.05\textwidth]{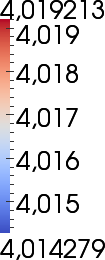}
    }
    \hfill
  }
  \caption{\label{fig:D=100}\footnotesize Instability for increased cytosolic
    diffusion ($D=100$). From left to right: the discrete solutions
    $u_h$ (upper row), $v_h$ (middle row) on level
    set $\{\phi_h = \frac1{2}\}$ for $t=0$, $t = 0.5235$, $t =
    1.0235$, and $t = 5.0235$, discrete solution $V_h$ on level
    set $\{\phi_h = \frac1{2}\}$ for $t=0$, $t = 5.0235$ (lower row).}
\end{figure}



\subsubsection{Stability for equal lateral and cytosolic diffusion
  coefficients} 

For $D = d = 1$, there is no instability as the results in
Fig. \ref{fig:D=1} indicate. 

\begin{figure}[here]
    \hfill
    \subfigure[$u_h(t = 0)$]{
      \includegraphics*[width=0.12\textwidth]{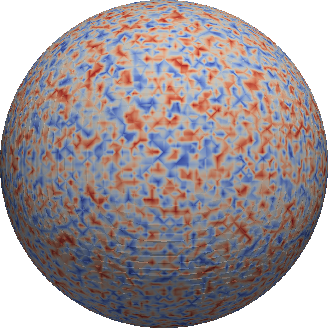}
    }
    \hfill
    \subfigure[]{
      \includegraphics*[width=0.04\textwidth]{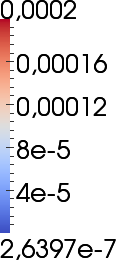}
    }
    \hfill
    \subfigure[$u_h(t = 5.0235)$]{
      \includegraphics*[width=0.12\textwidth]{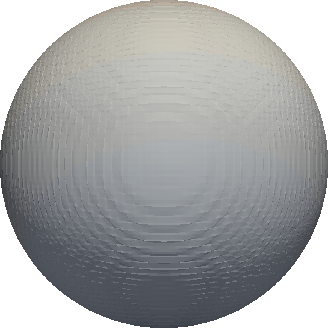}
    }
    \hfill
    \subfigure[]{
      \includegraphics*[width=0.04\textwidth]{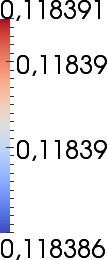}
    }
    \hfill
    \subfigure[$v_h(t = 0)$]{
      \includegraphics*[width=0.12\textwidth]{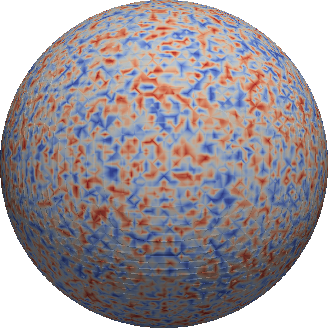}
    }
    \hfill
    \subfigure[]{
      \includegraphics*[width=0.04\textwidth]{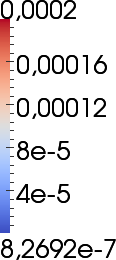}
    }
    \hfill
    \subfigure[$v_h(t = 5.0235)$]{
      \includegraphics*[width=0.12\textwidth]{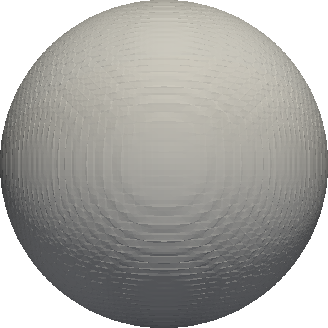}
    }
    \hfill
    \subfigure[]{
      \includegraphics*[width=0.04\textwidth]{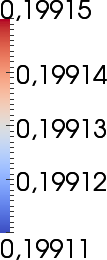}
    }
    \hfill\\
  \centerline{
    \hfill
    \subfigure[$V_h(t = 0)$]{
      \includegraphics*[width=0.12\textwidth]{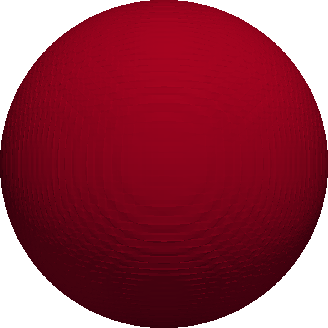}
    }
    \subfigure[]{
      \includegraphics*[width=0.028\textwidth]{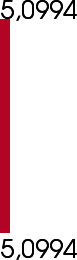}
    }
    \subfigure[$V_h(t = 5.0235)$]{
      \includegraphics*[width=0.12\textwidth]{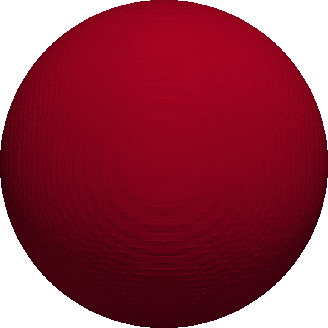}
    }
    \subfigure[]{
      \includegraphics*[width=0.04\textwidth]{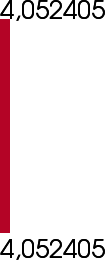}
    }
  \hfill
}
  \caption{\label{fig:D=1}\footnotesize Stability for equal membrane and
    cytosolic diffusion coefficients ($D=1$). From left to right: The
    discrete initial and stationary solutions $u_h$, $v_h$ and $V_h$
    on level set $\{\phi_h = \frac1{2}\}$.}
\end{figure}

\subsubsection{Ellipsoidal membrane}

We consider an ellipsoidal membrane $\Gamma$ with semi-axes $0.75$,
$1$ and $1.5$. As in the first example we have used $D = 100$. In
Fig. \ref{fig:D=100Ell}, the evolution towards a nearly stationary
discrete solution $u_h$ on the level set $\{\phi_h = \frac1{2}\}$ is
displayed. 

\begin{figure}[here]
  \centerline{
    \hfill
    \subfigure[$u_h(t = 0)$]{
      \includegraphics*[width=0.14\textwidth]{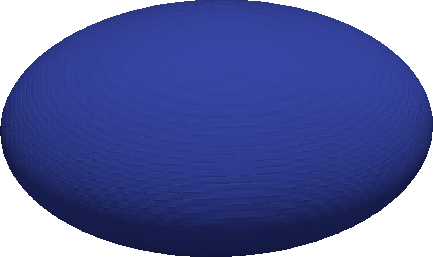}
    }
    \hfill
    \subfigure[$u_h(t = 2.0235)$]{
      \includegraphics*[width=0.14\textwidth]{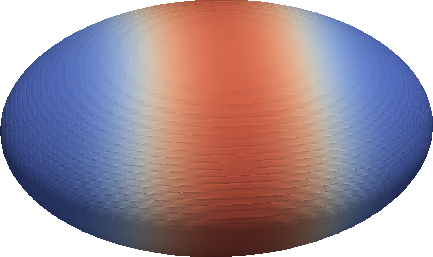}
    }
    \hfill
    \subfigure[$u_h(t = 12.0235)$]{
      \includegraphics*[width=0.14\textwidth]{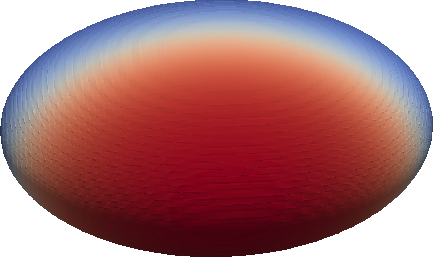}
    }
    \hfill
    \subfigure[$u_h(t = 27.0235)$]{
      \includegraphics*[width=0.14\textwidth]{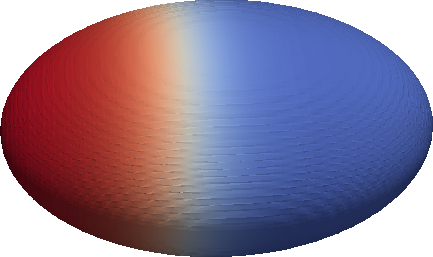}
    }
    \hfill
    \subfigure[]{
      \includegraphics*[width=0.05\textwidth]{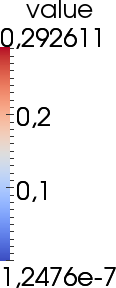}
    }
    \hfill
  }
  \caption{\label{fig:D=100Ell}\footnotesize Instability for increased cytosolic
    diffusion ($D=100$). From left to right: the discrete solutions
    $u_h$ on level set $\{\phi_h = \frac1{2}\}$ for $t=0$, $t = 2.0235$, $t =
    12.0235$, and $t = 27.0235$.}
\end{figure}

\section{Numerical Treatment of the Non-local System}\label{sec:numericsNonLocal}

In this section, we use a parametric finite element description of
the non-local system \eqref{eq:u-gen}--\eqref{eq:Vnew} in order to
numerically investigate instabilities for $d=1$ found in
Sec. \ref{sec:turingNonLocalA}. For this purpose, we apply the algorithm
described in \cite{RaRo12}. To discretize in time, we apply a
semi-implicit Euler scheme, where all nonlinearities are linearized in 
a suitable way. We use a parametric finite element approach
\cite{DzEl07a} with linear finite elements, where we solve as a system
for the two concentrations $u, v$ on the membrane. The non-local term
is treated fully explicitly. The resulting linear system is solved by
a stabilized bi-conjugate gradient method (BiCGStab). The scheme is
implemented using the adaptive finite element toolbox AMDiS
\cite{VeVo07}. 

\subsection{Numerical Examples}

In all following examples we use for $f$ and $q$ the specific choices
proposed in \eqref{eq:f-app} and \eqref{eq:q-app}, respectively. Thereby, we
use parameters from Table \ref{table2}. Furthermore, we consider the
unit-sphere $\Gamma = S^2$ and its discrete approximation $\Gamma_h$
through a triangulation with a uniform grid. We assume random initial
conditions $u_0: \Gamma_h \to [0,0.0002]$, $v_0: \Gamma_h \to
[0,0.0002]$. Moreover, we choose $V_{init} = 5.1$. Note that this choice of
initial conditions is the exact counterpart of the initial conditions
used for the simulation of the full system in
Section \ref{sec:numericsFull}. 

\begin{table}[h]
  \begin{center}
    \begin{tabular}[h]{| c | c | c | c | c | c | c |
        c | c | c | c | c | c | c | c | c | c | c |}
      \hline
      parameter & $d$ & $\gamma$ & $a_1$ &
      $a_2$ & $a_3$ 
      & $a_4$ & $a_5$
      & $a_6$  & $a_{-6}$ \\
      \hline
      value & $1$ & $400$ & $0.02$ & $20$ & $160$ &
      $1$ & $0.5$
      & $0.36$ & $5$ \\
      \hline
    \end{tabular}
    \vskip5pt \caption{\footnotesize Parameters used for numerical
      results (non-local model).}\label{table2}
  \end{center}
\end{table}

\subsubsection{Instability with equal lateral diffusion coefficients}

In Fig. \ref{fig:basic}, one can see contour plots of the discrete
solutions $u_h, v_h$ at different times for $d=1$. Similarly to the
first example in Sec. \ref{sec:numericsFull}, one observes an
evolution to an unstable spatially homogeneous solution and towards a
stationary solution with a single spot pattern, which is in agreement
with remarks in Sec.  \ref{sec:turingNonLocalA}.

\begin{figure}[here]
  \centerline{
    \hfill
    \subfigure[$u_h(t = 0)$]{
      \includegraphics*[width=0.14\textwidth]{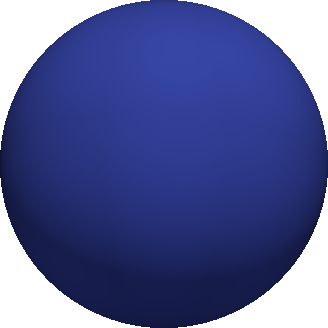}
    }
    \hfill
    \subfigure[$u_h(t = 0.5)$]{
      \includegraphics*[width=0.14\textwidth]{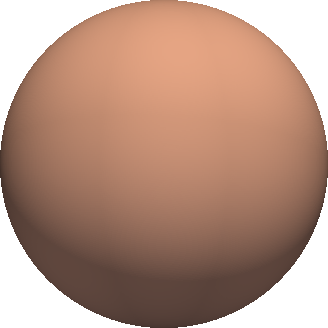}
    }
    \hfill
    \subfigure[$u_h(t = 1)$]{
      \includegraphics*[width=0.14\textwidth]{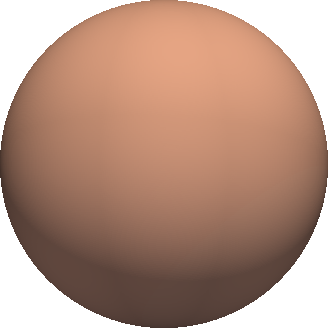}
    }
    \hfill
    \subfigure[$u_h(t = 5)$]{
      \includegraphics*[width=0.14\textwidth]{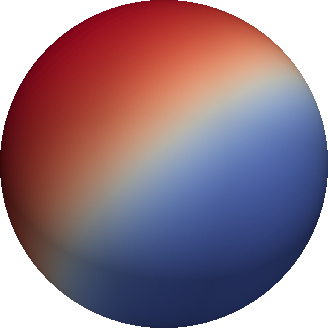}
    }
    \hfill
    \subfigure[]{
      \includegraphics*[scale=0.23]{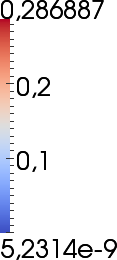}
    }
    \hfill
  }
  \vspace{3ex}
  \centerline{
    \hfill
    \subfigure[$v_h(t = 0)$]{
      \includegraphics*[width=0.14\textwidth]{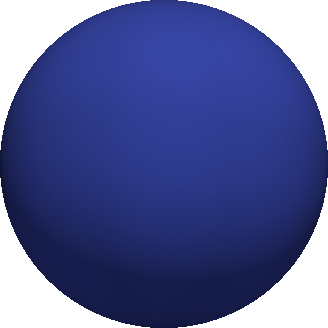}
    }
    \hfill
    \subfigure[$v_h(t = 0.5)$]{
      \includegraphics*[width=0.14\textwidth]{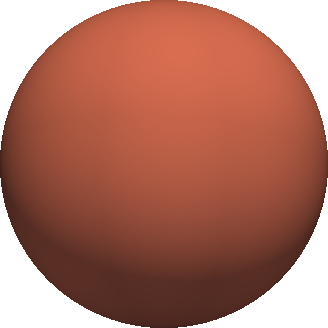}
    }
    \hfill
    \subfigure[$v_h(t = 1)$]{
      \includegraphics*[width=0.14\textwidth]{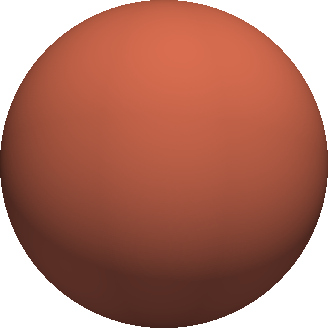}
    }
    \hfill
    \subfigure[$v_h(t = 5)$]{
      \includegraphics*[width=0.14\textwidth]{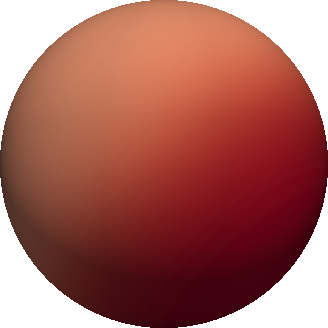}
    }
    \hfill
    \subfigure[]{
      \includegraphics*[scale=0.23]{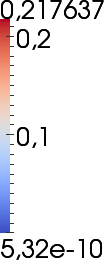}
    }
    \hfill
  }
  \caption{\label{fig:basic}\footnotesize Instability with diffusion ($d=1$). From
    left to right: the discrete solutions $u_h$ (upper row), $v_h$
    (lower row) for $t=0$, $t = 0.5$, $t = 1$, and $t = 5$.}
\end{figure}

\subsubsection{Stability for increased diffusion}

For increased diffusion coefficient, the instability vanishes, see
Fig. \ref{fig:d=10}. Here we have scaled $f$, $q$ and time $t$ by a
factor $1/10$, which corresponds to scaling the diffusion coefficients $d_u$ and
$d_v$ in the dimensional formulation (see the appendix) by a factor $10$. To be more precise, we have decreased the value of $\gamma$ from $400$ as in Table \ref{table2} to $40$ and have rescaled time. The explanation for the stabilization by lowering the value of $\gamma$ is that the inequality \eqref{eq:Ca2-2} is violated for small enough $\gamma$. In an informal way
one could explain this effect as a consequence of the `decreased difference' between cytosolic diffusion constant $D=\infty$ and lateral membrane diffusions $d_u,d_v$. This suggests, that for $d_u=d_v$ and an Case 2 instability large differences between $D$ and $d_u,d_v$ are required, which resembles a classical Turing type mechanism in the $V$, $u$ variables. 

\begin{figure}[here]
  \centerline{
    \hfill
    \subfigure[$u_h(t = 0)$]{
      \includegraphics*[width=0.12\textwidth]{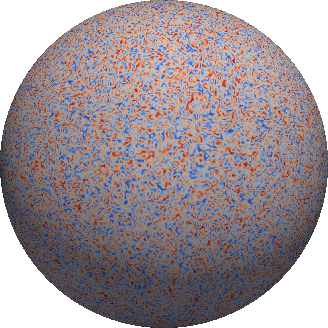}
    }
    \hfill
    \subfigure[]{
      \includegraphics*[width=0.04\textwidth]{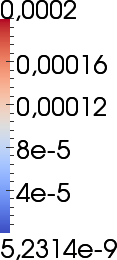}
    }
    \hfill    \subfigure[$u_h(t = 7)$]{
      \includegraphics*[width=0.12\textwidth]{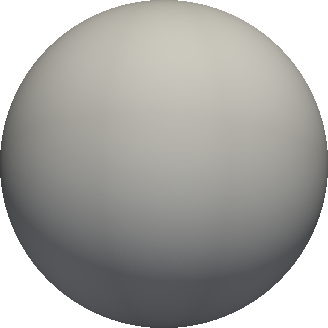}
    }
    \subfigure[]{
      \includegraphics*[width=0.04\textwidth]{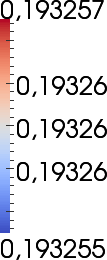}
    }
    \hfill
    \subfigure[$v_h(t = 0)$]{
      \includegraphics*[width=0.12\textwidth]{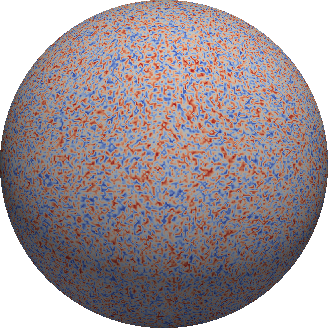}
    }
    \subfigure[]{
      \includegraphics*[width=0.04\textwidth]{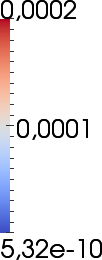}
    }
    \hfill
    \subfigure[$v_h(t = 7)$]{
      \includegraphics*[width=0.12\textwidth]{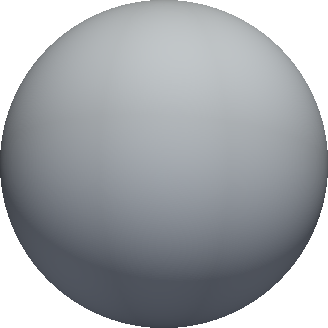}
    }
    \subfigure[]{
      \includegraphics*[width=0.04\textwidth]{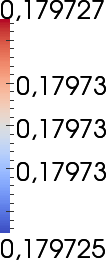}
    }
    \hfill
  }
  \caption{\label{fig:d=10}\footnotesize Stability for increased diffusion: The
    discrete initial and stationary solutions $u_h$ (left), $v_h$
    (right).}
\end{figure}

\subsubsection{Rich Nonlinear Dynamics}

In Fig. \ref{fig:richDynamics}, we present results showing the rich
dynamics the model includes. Thereby we replace the corresponding parameters in Table
\ref{table} by the values given in Table \ref{tab:richDynamics}. One observes that
the system evolves to a homogeneous stationary which is unstable and
later forms a pattern, which is again unstable. Finally the system
reaches a stable homogenous stationary state, different from the initial one.

\begin{table}[h!]
  \begin{center}
    \begin{tabular}[h]{| c | c | c | c | c | c | c |
        c | c | c | c | c | c | c | c | c | c | c | c |}
      \hline
      parameter & $d$ & $\gamma$ & $a_1$ &
      $a_2$ & $a_3$ 
      & $a_4$ & $a_5$
      & $a_6$  & $a_{-6}$ & $V_{init}$ \\
      \hline
      value & $1$ & $2000$ & $0.001$ & $20$ & $160$ &
      $1$ & $0.5$
      & $0.36$ & $10.3757$ & $10.1$ \\
      \hline
    \end{tabular}
    \vskip5pt \caption{\footnotesize Parameters used for numerical
      results in Fig. \ref{fig:richDynamics}.}\label{tab:richDynamics}
  \end{center}
\end{table}

\begin{figure}[here]
    \hfill
    \subfigure[\hspace{-13mm}$u_h(t = 0)$]{
      \includegraphics*[width=0.2\textwidth]{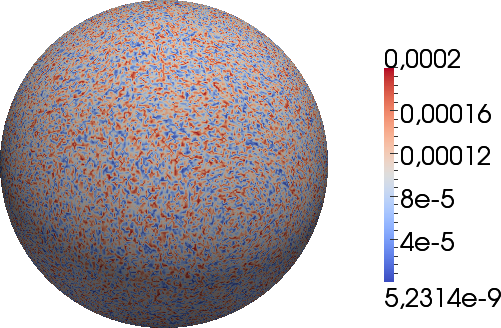}
    }
    \hfill
    \subfigure[\hspace{-11mm}$u_h(t = 1.85)$]{
      \includegraphics*[width=0.2\textwidth]{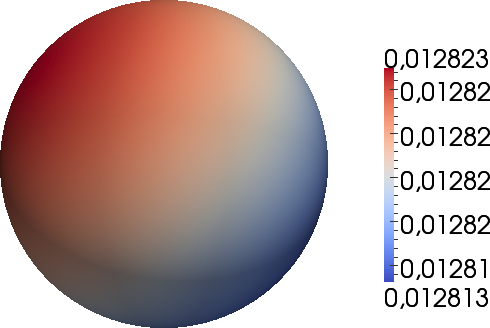}
    }
    \hfill
    \subfigure[\hspace{-11.5mm}$u_h(t = 3.2)$]{
      \includegraphics*[width=0.2\textwidth]{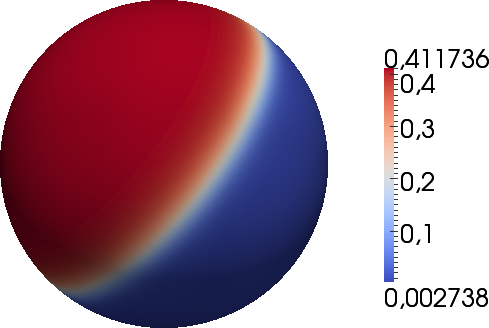}
    }
    \hfill
    \subfigure[\hspace{-11mm}$u_h(t = 3.45)$]{
      \includegraphics*[width=0.2\textwidth]{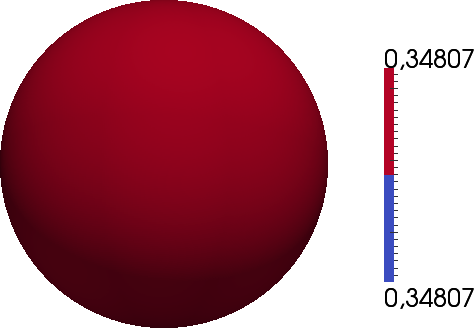}
    }
    \hfill
  \caption{\label{fig:richDynamics}\footnotesize Rich nonlinear dynamics for
    different set of parameters: discrete solution  $u_h$ at initial
    time, at $t=1.85$ showing an unstable intermediate solution, at
    $t=3.2$ with an intermediate pattern and homogenous stationary
    solution.}
\end{figure}

\section{Discussion}\label{sec:discuss}
We have analytically and numerically studied a coupled system of surface--bulk reaction--diffusion equations. Such a system for example arises in \cite{RaRo12} as a model for the GTPase cycle in biological cells. Here the relevant quantities are the concentrations $u,v$ of active and inactive GTPase on the membrane, and the concentration $V$ of inactive GTPase in the cytosolic volume. Our main interest here was to analyze the possibility of a symmetry breaking. The latter refers to an instability of spatially homogeneous stationary points that are stable with respect to symmetry-conserving perturbations, which in our case means that $u,v$, and the boundary values of $V$ are spatially homogeneous.

For spherical cell shapes we have performed a linearized stability analysis and have discovered that two different mechanisms for symmetry breaking are present. The first one requires a large difference in the lateral diffusion constants of $u$ and $v$, expressed by a large value of $d$ in \eqref{eq:V}--\eqref{eq:flux1}. This mechanism is closely related to a classical Turing instability for a two-variable system with $u$ as activator and $v$ as substrate. The second mechanism however does even occur for equal lateral diffusion constants of $u$ and $v$, i.e. even for $d=1$, which is much more realistic in the application to the GTPase cycle in cells. This second mechanism on the other hand requires that the cytosolic diffusion coefficient $D$ is much larger than the lateral diffusion coefficients. The simulations show that a factor of $100$ between both is sufficient, which is about the ratio between cytosolic and lateral membrane diffusion measured for biological cells \cite{KhHW00}. The new second instability mechanism is much closer to a Turing type mechanism in the $V,u$ variables and is a specific property of coupled surface--bulk reaction--diffusion systems.

Our results for the fully coupled system are comparable to those for its non-local reduction in the infinite cytosolic diffusion limit. However, here the corresponding stability analysis applies to more general membrane shapes and the reduction allows to further illuminate the two instability mechanisms. We in particular observe that for $d=1$ in the new second mechanism the most unstable perturbation is always in direction of eigenvectors corresponding to the smallest nonzero eigenvalue. This prefers the formation of a single membrane component with high concentration of activated GTPase. Again, this is much more in coincidence with the experimentally observed behavior than a classical Turing mechanism, where the emerging patterns are typically very sensitive to parameter changes. We remark that the same robust localization into a single spot has also been observed in \cite{GoPo08} in a related model for the GTPase cycle.

The distinction between the two mechanisms is expressed by the
inequalities \eqref{eq:Ca1-1} and \eqref{eq:Ca2-1}. As observed in
Remark \ref{rem:stab-ode} the latter implies that for Case 2
instabilities the spatially homogeneous state is unstable with respect
to spatially heterogeneous perturbations even in the absence of
membrane diffusion. It is important to notice that in the full and in
the reduced model there is a difference between (a) the stability with
respect to perturbations that are spatially homogeneous on the cell
membrane (and radially symmetric in the cell), and (b) the stability
(with respect to not necessarily spatially homogeneous perturbations)
of the corresponding systems without membrane diffusion. For local
reaction--diffusion systems such a difference is not present and there
is a coincidence between diffusive instabilities and symmetry
breaking. Since for our models and a Case 2 instability the symmetric
state is even without membrane diffusion unstable the term `diffusion
induced instability' might appear inappropriate. On the other hand, as
explained above, the instability origins from the large cytosolic
diffusion compared to the lateral diffusion and is in this respect
diffusion induced. The main property of both a Case 1 and a Case 2
instability is that we observe a symmetry breaking, which is clearly
confirmed by our numerical simulations. Starting from a spatially
homogeneous distribution, as long as no stationary point is reached,
the system is driven by the kinetic reaction and sorption terms and
spatial heterogeneities are not amplified. If the evolution approaches
a homogenous stationary state (and this typically requires its
stability with respect to spatially homogeneous perturbations) and if
this state is unstable with respect to general perturbations then
heterogeneous pattern develop and a symmetry breaking occurs. When the
evolution moves away from the stationary point the nonlinear effects
again come into play and determine the long-time behavior. 
\begin{appendix}

\section{Non-dimensionalization}

Here we recall the formulation of the reaction--diffusion model for the GTPase cycle proposed in \cite{RaRo12}. We give specific choices of reaction and attachment/detachment laws, and the dimensional formulation including all physical units.

As above we denote the cytosolic volume of a cell by $B$ and the cell membrane by the boundary $\Gamma=\partial B$ of $B$, which is assumed to represent a smooth, closed two-dimensional surface. In addition, we fix a time interval of observation $I$. We
formulate a system of reaction--diffusion equations for the unknowns
\begin{align*}
  V & \quad \text{concentration of cytosolic
    GDP-GTPase (in complex with GDI)},\\
  v &\quad \text{concentration of membrane-bound
    GDP-GTPase},\\
  u &\quad \text{concentration of membrane-bound
    GTP-GTPase}.
\end{align*}
Physical units are given by
\begin{equation*}
  [V] = \frac{\text{mol}}{\text{m}^3}, \quad 
  [u] = [v] = \frac{\text{mol}}{\text{m}^2}.
\end{equation*}
The following specific form was proposed in \cite{RaRo12},
\begin{align}
  \label{eq:V-dim}
  \pd_t V &= D \Delta V \quad \text{in} \quad B
  \times I,\\
  \label{eq:u-dim}
  \pd_t u &= d_u \SL u + k_1 v g_0\left(1 - \frac{K_5u}{1 +
      K_5 u}\right) + k_2 v \frac{K_5ug_0}{1 + K_5u} - k_3 \frac{u}{u +
    k_4}\quad \text{on} \quad
  \Gamma \times I,\\
  \label{eq:v-dim}
  \pd_t v &= d_v \SL v - k_1 v g_0\left(1 - \frac{K_5u}{1 +
      K_5 u}\right) - k_2 v \frac{K_5ug_0}{1 + K_5u} + k_3 \frac{u}{u
    + k_4} + q\quad \text{on} \quad \Gamma
  \times I
\end{align}
together with a constitutive law for attachment/detachment kinetics and a flux condition,
\begin{align}
  \label{eq:flux1-app}
  &-D \nabla V \cdot \nu = q  \quad \text{on} \quad \Gamma,\\
  \label{eq:flux2-app}
  &q = b_6 \frac{|B|}{|\Gamma|}V (c_{\max} - u - v)_+ - b_{-6}v.
\end{align}
The partial differential equations \eqref{eq:V-dim}--\eqref{eq:v-dim}
include kinetic rates $k_i$, $i \in \{1,2,3,4\}$, $K_5$ and an
equilibrium concentration of membrane-bound GEF $g_0$ with units
\begin{equation*}
  [k_1] = [k_2] = \frac{\text{m}^2}{\text{mol}\cdot\text{s}}; \quad
  [k_3] = \frac{\text{mol}}{\text{m}^2\text{s}} ; \quad
  [k_4] = \frac{\text{mol}}{\text{m}^2} ; \quad
  [K_5] = \frac{\text{m}^2}{\text{mol}} ; \quad
  [g_0] = \frac{\text{mol}}{\text{m}^2}.
\end{equation*}
Furthermore, we have units of diffusion coefficients $D, d_u, d_v$ and
sorption coefficients $b_6, b_{-6}$ given by 
\begin{equation*}
  [D] = [d_u] = [d_v] = \frac{\text{m}^2}{\text{s}}; \quad 
  [b_6] = \frac{\text{m}^2}{\text{mol} \cdot \text{s}}; \quad 
  [b_{-6}] = \frac{1}{\text{s}}.
\end{equation*}
In \eqref{eq:flux1-app}, $\nu$ denotes the outward unit normal to $B$
at $\Gamma$. The constitutive equation \eqref{eq:flux2-app} for the
flux $q$ includes a saturation value $c_{\max}$. Membrane attachment is
treated as a reaction between cytosolic GTPase and a free site on the
membrane and modeled by a Langmuir rate law \cite{Kell09}. Detachment
is taken proportional to the inactive GTPase concentration.

In order to obtain a non-dimensional model, we follow \cite{RaRo12} and
define dimensionless spatial and time coordinates
\begin{align*}
  \xi := \frac1{R}x; \quad   \tau := \frac{d_u}{R ^2} t,
\end{align*}
where $R > 0$ denotes a typical length. We define $\gamma>0$ through
$R=\sqrt{\gamma}\,\mathbb{I}$ with $\mathbb{I}=1\text{m}$ denoting the
unit length. This leads to transformed domain $\tilde B := \{\xi \in
\R^3 : R \xi \in B \}$, $\tilde \Gamma := \pd \tilde B$ and time interval $\tilde
I := \{ \tau \in \R : \frac{R^2\tau}{d_u} \in I\}$. Non-dimensional
concentrations are defined through 
\begin{equation*}
  \tilde V := \frac{R}{c_{\max}} V , \quad 
  \tilde u := \frac{1}{c_{\max}} u, \quad
  \tilde v := \frac{1}{c_{\max}} v. 
\end{equation*}
Moreover, we introduce dimensionless quantities
\begin{gather*}
  a_1 := \frac{\mathbb{I} ^2}{d_u}k_1g_0, \quad 
  a_2 := \frac{1}{K_5c_{\max}}, \quad
  a_3 := \frac{k_2}{k_1}a_1, \quad
  a_4 := \frac{\mathbb{I} ^2}{d_u c_{\max}} k_3, \quad
  a_5 := \frac{k_4}{c_{\max}}, \\
  a_6 := \frac{\mathbb{I}^2 b_6}{d_u} c_{\max}\frac{|B|}{|\Gamma|R},\quad
  a_{-6} := \frac{\mathbb{I} ^2 b_{-6}}{d_u},\quad
  d := \frac{d_v}{d_u}, \quad
  \tilde D := \frac{D}{d_u}.
\end{gather*}
With these definitions, dropping all tildes and replacing
$\xi$ and $\tau$ by $x$ and $t$, respectively, we arrive at the full
mathematical model \cite{RaRo12}

\begin{align}
  \label{eq:V-app}
  \pd_t  V &=  D \Delta V \quad \text{in} \quad
  B  \times  I,\\
  \label{eq:u-app}
  \pd_t  u &= \SL  u + \gamma f(u,v)\quad \text{on} \quad
  \Gamma \times  I,\\
  \label{eq:v-app}
  \pd_t  v &= d \SL  v + \gamma (- f(u,v) + q(u,v,V)) \quad \text{on}
  \quad  \Gamma \times  I
\end{align}
with the flux condition
\begin{equation}
  \label{eq:flux-app}
  - D \nabla  V \cdot  \nu = \gamma  q \quad
  \text{on} \quad  \Gamma \times  I, 
\end{equation}
where
\begin{align}
  \label{eq:f-app}
  f(u,v)\,&:=\, 
  \left(a_1 + (a_3 -
    a_1)\frac{ u}{a_2 +  u} \right) v - a_4
  \frac{ u}{a_5 + u},\\
  \label{eq:q-app}
  q(u,v,V) \,&:=\, 
  a_6  V (1 - ( u +  v))_+ - a_{-6}  v.
\end{align}
We remark that for $a_1<a_3$, which is again a natural assumption (see \cite{RaRo12}), all the properties  that we have assumed in \eqref{cdt:q} are satisfied.
\end{appendix}

\end{document}